\documentclass{pnastwo_mod}

\usepackage{graphicx}  

\usepackage{multibib}
\newcites{pri}{Primary sources}
\newcites{sec}{Secondary sources}

\usepackage{amssymb,amsfonts,amsmath}

\usepackage{graphicx}          		
\usepackage{graphics}         	 	
\usepackage{epsfig}            		
\usepackage{subfigure}        		
\usepackage{bm}
\usepackage[usenames,dvipsnames]{color}   

\def\EndProof{ \quad \vrule width 1.3ex height 1.3ex depth 0pt }
\newenvironment{proof}{\textbf{Proof}\hspace{4pt}}{\hfill\EndProof}

\newcommand\oprocendsymbol{\hbox{$\square$}}
\newcommand\oprocend{\relax\ifmmode\else\unskip\hfill\fi\oprocendsymbol}

\newcounter{example}
    \newenvironment{example}[1][]{\refstepcounter{example}\par\medskip\noindent%
       \textbf{Example~\theexample. #1} \rmfamily}{\medskip}

\newcommand*\fvec[1]{\ensuremath{\mathbf{#1}}}                                	
\newcommand*\dt[0]{\frac{d}{d\,t}\,}					                      	
\newcommand*\mc[0]{\mathcal}                                                  		
\newcommand*\mbb[0]{\mathbb}                                                  		
\DeclareMathOperator*{\argmax}{argmax}                                        		
\DeclareMathOperator*{\diag}{\mathrm{diag}}						

\newcommand{\Ker}{\mathrm{Ker\,}}
\newcommand{\Img}{\mathrm{Im\,}}
\DeclareMathOperator*{\sinbf}{\mathrm{\bf sin}} 	                              
\DeclareMathOperator*{\arcsinbf}{\mathrm{\bf arcsin}} 	                              

\newcommand{\until}[1]{\{1,\dots, #1\}}
\newcommand{\subscr}[2]{#1_{\textup{#2}}}

\newcommand{\setdef}[2]{\{#1 \; | \; #2\}}

\newcommand{\map}[3]{#1: #2 \rightarrow #3}

\newcommand{\real}{\mathbb{R}}

\newcommand{\mycircle}{\ensuremath{\mbb S^1}}
\newcommand{\torus}{\ensuremath{\mbb T}}
\newcommand{\rot}{\operatorname{rot}}

\contributor{Submitted to Proceedings
of the National Academy of Sciences of the United States of America}
\url{www.pnas.org/cgi/doi/10.1073/pnas.0709640104}
\copyrightyear{2008}
\issuedate{Issue Date}
\volume{Volume}
\issuenumber{Issue Number}

\contributor{}
\url{}
\copyrightyear{}
\issuedate{}
\volume{}
\issuenumber{}

\begin{document}



\title{Synchronization in Complex Oscillator Networks and Smart Grids}





\author{Florian D\"orfler\affil{1}{Center for Control, Dynamical Systems and Computation, University of California at Santa Barbara, Santa Barbara, CA 93106, USA}\affil{2}{Center for Nonlinear Studies and Theory Division, Los Alamos National Laboratory, NM 87545, USA},
Michael Chertkov\affil{2}{},
\and
Francesco Bullo\affil{1}{}}

\contributor{}

\maketitle

\begin{article}

\begin{abstract}
The emergence of synchronization in a network of coupled oscillators is a fascinating topic in various scientific disciplines. A coupled oscillator network is characterized by a population of heterogeneous oscillators and a graph describing the interaction among them. It is known that a strongly coupled and sufficiently homogeneous network synchronizes, but the exact threshold from incoherence to synchrony is unknown. Here we present a novel, concise, and closed-form condition for synchronization of the fully nonlinear, non-equilibrium, and dynamic network. Our synchronization condition can be stated elegantly in terms of the network topology and parameters, or equivalently in terms of an intuitive, linear, and static auxiliary system.  Our results significantly improve upon the existing conditions advocated thus far, they are provably exact for various interesting network topologies and parameters, they are statistically correct for almost all networks, and they can be applied equally to synchronization phenomena arising in physics and biology as well as in engineered oscillator networks such as electric power networks. We illustrate the validity, the accuracy, and the practical applicability of our results in complex networks scenarios and in smart grid applications.
\end{abstract}






The scientific interest in the synchronization of coupled oscillators can be traced back to Christiaan Huygens' seminal work on ``an odd kind sympathy'' between coupled pendulum clocks \citepri{SI-CH:1673}, and it continues to fascinate the scientific community to date \citepri{SI-SHS:03,SI-ATW:01}. A mechanical analog of a coupled oscillator network is shown in Figure \ref{Fig: Mechanical analog} and consists of a group of particles constrained to rotate around a circle and assumed to move without colliding. Each particle is characterized by a phase angle $\theta_{i}$ and has a preferred natural rotation frequency $\omega_{i}$. Pairs of interacting particles $i$ and $j$ are coupled through an elastic spring with stiffness $a_{ij}$. 
Intuitively, a weakly coupled oscillator network with strongly heterogeneous natural frequencies $\omega_{i}$ does not display any coherent behavior, whereas a strongly coupled network with sufficiently homogeneous natural frequencies is amenable to synchronization. These two qualitatively distinct regimes are illustrated in Figure~\ref{Fig: Mechanical analog}.
\begin{figure}[htbp]
	\centering{
	\includegraphics[width=0.96\columnwidth]{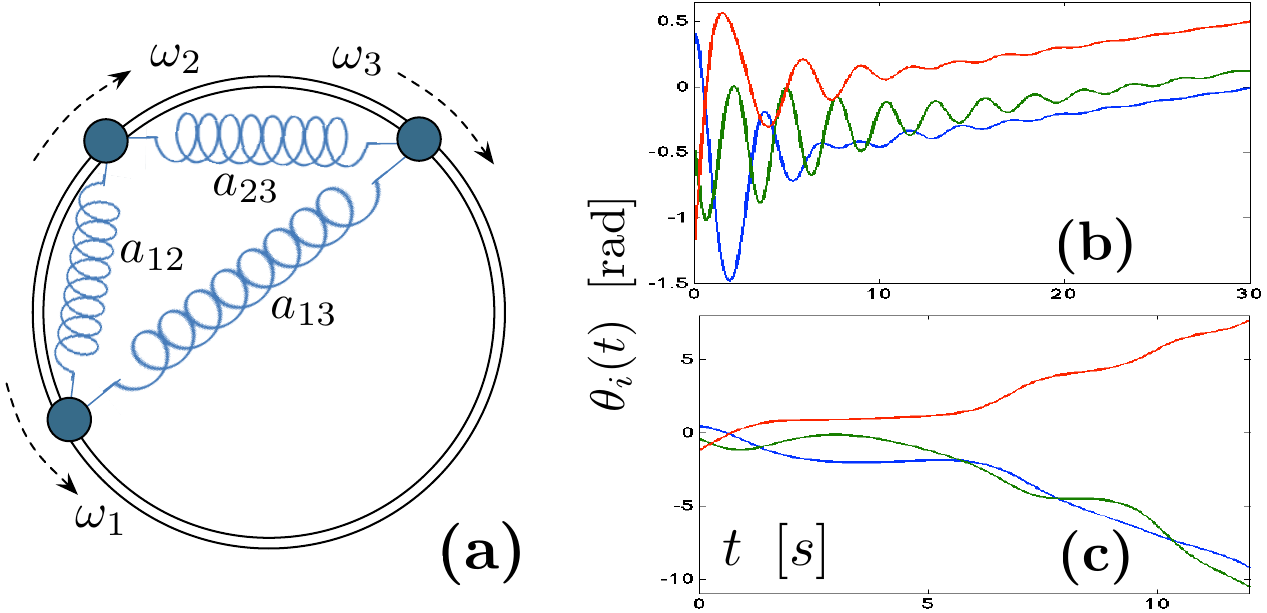}
	\caption{Mechanical analog of a coupled oscillator network (a) and its dynamics in a strongly coupled (b) and weakly coupled (c) network. With exception of the coupling weights $a_{ij}$, all parameters in the simulation (b) and (c) are identical.}
	\label{Fig: Mechanical analog}
	}
\end{figure}

Formally, the interaction among $n$ such phase oscillators is
modeled by a connected graph $G(\mc V,\mc E,A)$ with nodes $\mc V = \until
n$, edges $\mc E \subset \mc V \times \mc V$, and positive weights
$a_{ij}>0$ for each undirected edge $\{i,j\} \in \mc E$. For pairs of non-interacting oscillators $i$ and $j$, the coupling weight $a_{ij}$ is zero. 
We assume that the node set is partitioned as $\mc V = \mc V_{1} \cup \mc V_{2}$, and we consider the following general coupled oscillator model:
\begin{align}
	\begin{split}
	M_{i} \ddot \theta_{i} + D_{i} \dot \theta_{i}
	&=
	\omega_{i} - \sum\nolimits_{j=1}^{n} a_{ij} \sin(\theta_{i}-\theta_{j})
	\,,\qquad i \in \mc V_{1}\,,
	\\
	D_{i} \dot \theta_{i}
	&=
	\omega_{i} - \sum\nolimits_{j=1}^{n} a_{ij} \sin(\theta_{i}-\theta_{j})
	\,,\qquad i \in \mc V_{2} \,.
	\end{split}
	\label{eq: coupled oscillator model}
\end{align}
The coupled oscillator model \eqref{eq: coupled oscillator model} consists
of the second-order oscillators $\mc V_{1}$ with Newtonian dynamics, inertia
coefficients $M_{i}$, and viscous damping $D_{i}$. The remaining 
oscillators $\mc V_{2}$ feature first-order dynamics with time constants $D_{i} $.
A perfect electrical analog of the coupled oscillator model \eqref{eq: coupled oscillator model} is given by the classic structure-preserving power network model \citepri{SI-ARB-DJH:81}, our enabling application of interest. Here, the first and second-order dynamics correspond to loads and generators, respectively, and the right-hand sides depict the power injections $\omega_{i}$ and the power flows $a_{ij} \sin(\theta_{i}-\theta_{j})$ along transmission~lines.

The rich dynamic behavior of the coupled oscillator model \eqref{eq:
  coupled oscillator model} arises from a competition between each
oscillator's tendency to align with its natural frequency $\omega_{i}$ and
the synchronization-enforcing coupling $a_{ij} \sin(\theta_{i} -
\theta_{j})$ with its neighbors.
In absence of the first term, the coupled oscillator dynamics \eqref{eq: coupled oscillator model} collapse to a trivial phase-synchronized equilibrium, where all angles $\theta_{i}$ are aligned. The dissimilar natural frequencies $\omega_i$, on the other hand, drive the oscillator network away from this all-aligned equilibrium. Moreover, even if the coupled oscillator model \eqref{eq: coupled oscillator model} synchronizes, it still carries the flux of angular rotation,  
respectively, the flux of electric power from generators to loads in a power network.
The main and somehow surprising result of this paper is that, in spite of all the aforementioned complications,  an elegant and easy to verify criterion characterizes synchronization of the nonlinear and non-equilibrium dynamic oscillator network \eqref{eq: coupled oscillator model}.

\section*{Review of Synchronization in Oscillator Networks}

The coupled oscillator model \eqref{eq: coupled oscillator model} unifies various models in the literature including dynamic models of electric power networks. 
The supplementary information (SI) discusses modeling of electric power networks in detail.
For $\mc V_{2} = \emptyset$, the coupled oscillator model \eqref{eq: coupled oscillator model} appears in synchronization phenomena in animal flocking behavior \citepri{SI-SYH-EJ-MJK:10},
populations of flashing fireflies \citepri{SI-GBE:91}, crowd synchrony on London's Millennium bridge \citepri{SI-SHS-DMA-AMR-BE-EO:05}, 
as well as Huygen's pendulum clocks \citepri{SI-MB-MFS-HR-KW:02}. 
For $\mc V_{1}=\emptyset$, the coupled oscillator model \eqref{eq: coupled oscillator model} reduces to the celebrated Kuramoto model \citepri{SI-YK:75}, which appears in coupled Josephson junctions \citepri{SI-KW-PC-SHS:98}, particle coordination \citepri{SI-DAP-NEL-RS-DG-JKP:07}, spin glass models~\citepri{SI-GJ-JA-DB-ACCC-CPV:01,SI-HD:92}, 
neuroscience~\citepri{SI-FV-JPL-ER-JM:01}, deep brain stimulation \citepri{SI-PAT:03}, 
chemical oscillations \citepri{SI-IZK-YZ-JLH:02}, biological locomotion \citepri{SI-NK-GBE:88}, rhythmic applause \citepri{SI-ZN-ER-TV-YB-AIB:00}, and countless other 
synchronization phenomena \citepri{SI-SHS:00,SI-JAA-LLB-CJPV-FR-RS:05,SI-FD-FB:10w}. Finally, coupled oscillator models of the form \eqref{eq: coupled oscillator model} also serve as prototypical examples in complex networks studies~\citepri{SI-AA-ADG-JK-YM-CZ:08,SI-SB-VL-YM-MC-DUH:06}. 

The coupled oscillator dynamics \eqref{eq: coupled oscillator model} feature
 the synchronizing effect of the coupling described by the graph $G(\mc
V,\mc E,A)$ and the de-synchronizing effect of the dissimilar natural
frequencies $\omega_{i}$. 
The complex network community asks questions of the form ``what are the conditions on the coupling and the dissimilarity such that a synchronizing behavior emerges?'' Similar questions appear also in all the aforementioned applications, for instance, in large-scale electric power systems. 
Since synchronization is pervasive in the operation of an interconnected power grid, a central question is ``under which conditions on the network parameters and topology,
the current load profile and power generation, does there
exist a synchronous operating point \citepri{SI-BCL-PWS-MAP:99,SI-ID:92}, when is it optimal \citepri{SI-JL-DT-BZ:10}, when is it stable 
\citepri{SI-DJH-GC:06,SI-FD-FB:09z}, and how robust is it \citepri{SI-MI:92,AA-SS-VP:81,SI-FW-SK:82,SI-FFW-SK:80}?'' 
A local loss of synchrony can trigger cascading failures and possibly result in wide-spread blackouts.
In the face of the complexity of future smart grids and the integration
challenges posed by renewable energy sources, a deeper understanding of
synchronization is increasingly important.

Despite the vast scientific interest, the search for
sharp, concise, and closed-form synchronization conditions for coupled
oscillator models of the form \eqref{eq: coupled oscillator model} has been
so far in vain.
Loosely speaking, synchronization occurs when the
coupling dominates the dissimilarity.  Various conditions have been proposed to quantify this trade-off ~\citepri{SI-FD-FB:10w,SI-FFW-SK:80,SI-FD-FB:09z,SI-AJ-NM-MB:04,SI-AA-ADG-JK-YM-CZ:08,SI-SB-VL-YM-MC-DUH:06,SI-FW-SK:82,SI-LB-LS-ADG:09}.
The coupling is typically quantified by the nodal degree or the algebraic connectivity of the graph $G$, 
and the
dissimilarity is quantified by the magnitude or the spread of the natural frequencies $\omega_{i}$.
Sometimes, these conditions can be evaluated only numerically since they depend on the network state \citepri{SI-FFW-SK:80,SI-FW-SK:82} or arise from a non-trivial linearization process, such as the Master stability function formalism
\citepri{SI-AA-ADG-JK-YM-CZ:08,SI-SB-VL-YM-MC-DUH:06}.
To date, exact synchronization conditions are known only for simple coupling topologies~\citepri{SI-NK-GBE:88,SI-FD-FB:10w,SI-SHS-REM:88,SI-MV-OM:09}. For arbitrary topologies only sufficient conditions are known \citepri{SI-FFW-SK:80,SI-FD-FB:09z,SI-AJ-NM-MB:04,SI-FW-SK:82} as well as numerical investigations for random networks~\citepri{SI-JGG-YM-AA:07,SI-TN-AEM-YCL-FCH:03,SI-YM-AFP:04}. 
Simulation studies indicate that the known sufficient conditions are
very conservative estimates on the threshold from incoherence to synchrony.
Literally, every review article on synchronization concludes emphasizing
the quest for exact synchronization conditions for arbitrary network
topologies and
parameters \citepri{SI-JAA-LLB-CJPV-FR-RS:05,SI-FD-FB:10w,SI-SHS:00,SI-AA-ADG-JK-YM-CZ:08,SI-SB-VL-YM-MC-DUH:06}.
In this article, we present a concise and sharp synchronization condition which features elegant graph-theoretic and physical interpretations.

\section*{Novel Synchronization Condition}

For the coupled oscillator model \eqref{eq: coupled oscillator model} and
its applications, the following notions of
synchronization are appropriate. First, a solution has {\em synchronized
  frequencies} if all frequencies $\dot \theta_{i}$ are identical to a
common constant value $\subscr{\omega}{sync}$.
If a synchronized solution exists, it is known that the synchronization
frequency is $\subscr{\omega}{sync} = \sum_{k=1}^{n} \omega_{k} /
\sum_{k=1}^{n} D_{k}$ and that, by working in a rotating reference frame,
one may assume $\subscr{\omega}{sync} = 0$.
Second, a solution has {\em cohesive phases} if every pair of connected
oscillators has phase distance smaller than some angle $\gamma \in
{[0,\pi/2[}$, that is, $|\theta _{i} - \theta _{j}| < \gamma$ for every
    edge $\{i,j\} \in \mc E$.

Based on a novel analysis approach to the synchronization problem, we
propose the following synchronization condition for the coupled oscillator
model \eqref{eq: coupled oscillator model}:
\begin{quotation}
  \noindent {\bf Sync condition:} The coupled oscillator model \eqref{eq: coupled
    oscillator model} has a unique and stable solution $\theta^{*}$ with
  synchronized frequencies and cohesive phases $|\theta _{i}^{*} - \theta
  _{j}^{*}| \leq \gamma < \pi/2$ for every pair of connected oscillators
  $\{i,j\} \in \mc E$~if
  \begin{equation}
    \bigl\| L^{\dagger} \omega \bigr\|_{\mc E,\infty} \leq \sin(\gamma) \,.
    \label{eq: sync condition - gamma}
  \end{equation}
Here, $L^{\dagger}$ is the pseudo-inverse of the network Laplacian
matrix $L$
 and $\left\| x \right\|_{\mc E,\infty} = \max_{\{i,j\} \in \mc E} |
x_{i} - x_{j}|$ is the worst-case dissimilarity for 
$x\!=\!(x_1,\dots,x_n)$ over the edges $\mc E$.
\end{quotation}
%
We establish the broad applicability of the proposed condition \eqref{eq: sync condition - gamma} to various
classes of networks via analytical and statistical methods in the next section.
Before that,  we provide some equivalent formulations for condition [2] in order to develop deeper intuition and obtain insightful conclusions.

{\bf Complex network interpretation:} Surprisingly, topological or spectral
connectivity measures such as nodal degree or algebraic connectivity are
not key to synchronization. In fact, these often advocated
\citepri{SI-FFW-SK:80,SI-FD-FB:09z,SI-AJ-NM-MB:04,SI-FW-SK:82,SI-AA-ADG-JK-YM-CZ:08,SI-SB-VL-YM-MC-DUH:06}
connectivity measures turn out to be conservative estimates of the
synchronization condition \eqref{eq: sync condition - gamma}. This
statement can be seen by introducing the matrix $U$ of orthonormal
eigenvectors of the network Laplacian matrix $L$ with corresponding eigenvalues
$0=\lambda_1<\lambda_2\leq\cdots\leq\lambda_n$.  From this spectral
viewpoint, condition \eqref{eq: sync condition - gamma} can be equivalently
written as
\begin{equation}
  \bigl\| U \diag\bigl( 0, 1/\lambda_2, \dots, 1/\lambda_n  \bigr)  \cdot \bigl( U^{T} \omega \bigr) \bigr\|_{\mc E,\infty} \leq \sin(\gamma)
  \label{eq: equivalent modal condition}
  \,.
\end{equation}
In words, the natural frequencies $\omega$ are projected on the network
modes $U$, weighted by the inverse Laplacian eigenvalues, and $\| \cdot
\|_{\mc E,\infty}$ evaluates the worst-case dissimilarity of this weighted
projection.
A sufficient condition for the inequality \eqref{eq: equivalent modal condition} to be true is the algebraic connectivity condition
$\lambda_{2} \geq \| \omega \|_{\mc E,\infty} \cdot \sin(\gamma)$.
Likewise, a necessary condition for inequality \eqref{eq: equivalent modal condition} is $2 \cdot \textup{deg}(G) \geq \lambda_{n} \geq \| \omega \|_{\mc E,\infty} \cdot \sin(\gamma)$, where $\textup{deg}(G)$ is the maximum nodal degree in the graph $G(\mc V,\mc E,A)$.
Clearly, when compared to \eqref{eq: equivalent modal condition}, this
sufficient condition and this necessary condition feature only one of $n-1$
non-zero Laplacian eigenvalues and are overly conservative.

{\bf Kuramoto oscillator perspective:}
Notice, that in the limit $\gamma\to\pi/2$, condition \eqref{eq: sync condition - gamma} suggests that there exists a stable synchronized solution if 
\begin{equation}
	\bigl\| L^{\dagger} \omega \bigr\|_{\mc E,\infty} < 1
	\label{eq: sync condition}
	\,.
\end{equation}
For classic Kuramoto oscillators coupled in a complete graph with uniform weights $a_{ij} = K/n$, the synchronization condition \eqref{eq: sync condition} reduces to the condition $K > \max_{i,j \in \until n} |\omega_{i} - \omega_{j}|$, known for the classic Kuramoto model \citepri{SI-FD-FB:10w}.

{\bf Power network perspective:}
In power systems engineering, the equilibrium equations of the coupled oscillator model \eqref{eq: coupled oscillator model}, given by $\omega_{i} = \sum_{j=1}^{n} a_{ij}\sin(\theta_{i} - \theta_{j})$, are referred to as the AC power flow equations, and they are often approximated by their linearization \citepri{SI-MI:92,SI-AA-SS-VP:81,SI-FW-SK:82,SI-FFW-SK:80} $\omega_{i} = \sum_{j=1}^{n} a_{ij}(\theta_{i} - \theta_{j})$, known as the DC power flow equations. In vector notation the DC power flow equations read as $\omega = L \theta$, and their solution satisfies $\max_{\{i,j\} \in \mc E} | \theta_{i} - \theta_{j}| = \| L^{\dagger} \omega \|_{\mc E,\infty}$. According to condition \eqref{eq: sync condition - gamma}, the worst phase distance $\| L^{\dagger} \omega \|_{\mc E,\infty}$ obtained by the DC power flow equations needs to be smaller than $\sin(\gamma)$, such that the solution to the AC power flow equations satisfies $\max_{\{i,j\} \in \mc E} |\theta_{i} - \theta_{j}| < \gamma$. 
Hence, our condition extends the common DC power flow approximation from infinitesimally small angles $\gamma \ll 1$ to large angles $\gamma \in {[0,\pi/2[}$.

{\bf Auxiliary linear perspective:}
As detailed in the previous paragraph, the key term $L^{\dagger} \omega$ in condition \eqref{eq: sync condition - gamma} equals the phase differences obtained by the linear Laplacian equation $\omega = L \theta$. This {\em linear} interpretation is not only insightful but also practical since condition \eqref{eq: sync condition - gamma} can be quickly evaluated by numerically solving the sparse linear system $\omega = L \theta$. Despite this linear interpretation, we emphasize that our derivation of condition \eqref{eq: sync condition - gamma} is not based on any linearization arguments.

{\bf  Energy landscape perspective:}
Condition \eqref{eq: sync condition - gamma} can also be understood in terms of an appealing energy landscape interpretation. The coupled oscillator model \eqref{eq: coupled oscillator model} is a system of particles that aim to minimize the energy function
\begin{equation*}
	E(\theta)
	= 
	\sum\nolimits_{\{i,j\} \in \mc E} a_{ij} \bigl( 1 - \cos(\theta_{i} - \theta_{j} ) \bigr) - \sum\nolimits_{i=1}^{n} \omega_{i} \cdot \theta_{i}
	\,,
\end{equation*}
where the first term is a pair-wise nonlinear attraction among the particles, and the second term represents the external force driving the particles away from the ``all-aligned" state. 
Since the energy function $E(\theta)$ is difficult to study, it is natural
to look for a minimum of its second-order approximation
$E_0(\theta)=\sum_{\{i,j\} \in \mc E} a_{ij} (\theta_i-\theta_j)^2/2 -
\sum_{i=1}^{n} \omega_i \cdot \theta_i$, where the first term corresponds
to a Hookean potential. Condition \eqref{eq: sync condition - gamma} is then
restated as follows: $E(\theta)$ features a phase cohesive minimum 
with interacting particles no further than $\gamma$ apart if
$E_{0}(\theta)$ features a minimum with interacting particles no further
from each other than $\sin(\gamma)$, as illustrated in Figure \ref{Fig: Energy landscape}.
\begin{figure}[htbp]
	\centering{
	\includegraphics[width=0.85\columnwidth]{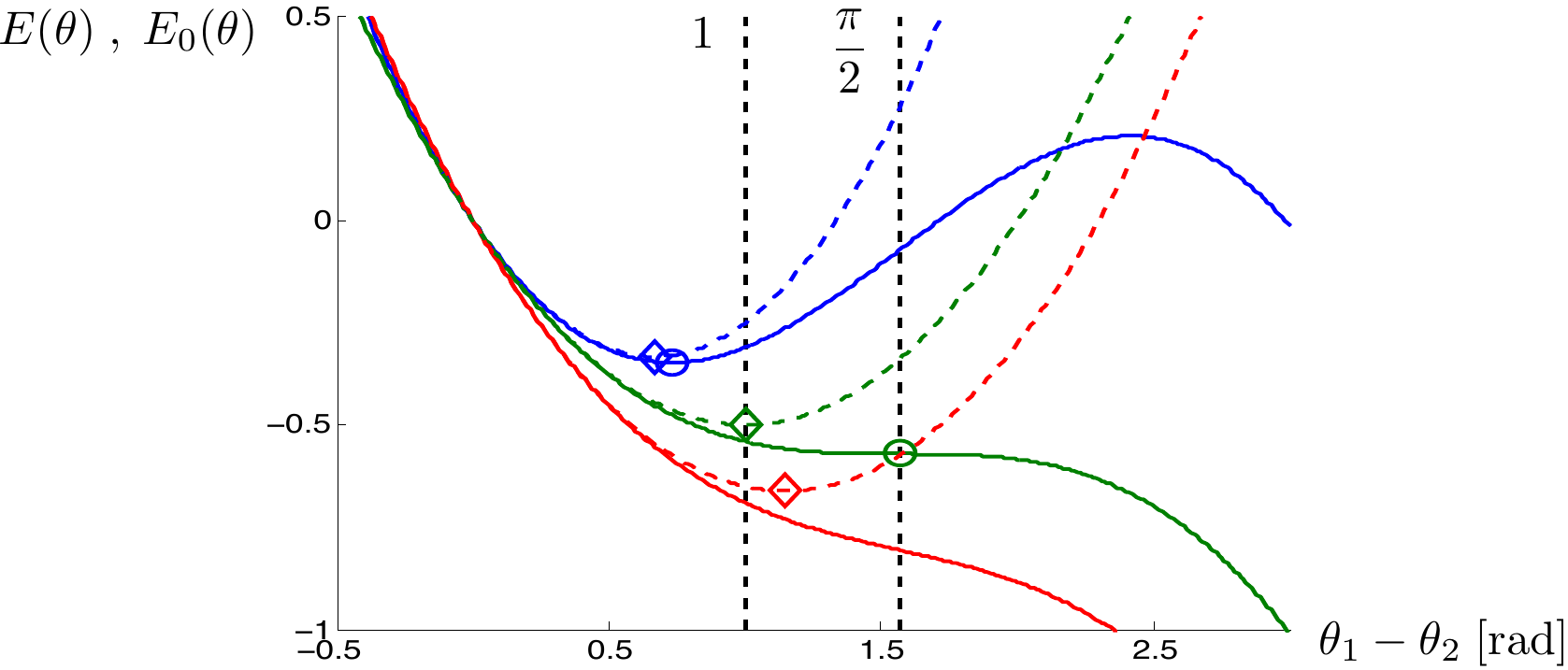}
	\caption{The energy function $E(\theta)$ and its quadratic approximation $E_{0}(\theta)$ for a two-particle system are shown as solid and dashed curves, respectively, for the stable (blue), marginal (green) and unstable (red) cases. The circles and diamonds represent stable critical points of $E(\theta)$ and $E_{0}(\theta)$.} 
	\label{Fig: Energy landscape}
	}
\end{figure}

\section*{Analytical and Statistical Results}

Our analysis approach to the synchronization problem is based on algebraic graph theory. We propose an equivalent reformulation of the synchronization problem, which reveals the crucial role of cycles and cut-sets in the graph and ultimately leads to the synchronization condition \eqref{eq: sync condition - gamma}.
%
%
In particular, we analytically establish  the synchronization condition \eqref{eq: sync condition - gamma} for the following six interesting cases: 
\begin{quotation}
  \noindent {\bf Analytical result:} The synchronization condition \eqref{eq: sync condition - gamma} is necessary and sufficient for {\em (i)} the sparsest (acyclic) and {\em (ii)} the densest (complete and uniformly weighted) network topologies $G(\mc V,\mc E,A)$, {\em (iii)} the best (phase synchronizing) and {\em (iv)} the worst (cut-set inducing) natural frequencies, {\em (v)} for cyclic topologies of length strictly less than five, {\em (vi)} for arbitrary cycles with symmetric parameters, {\em (vii)} as well as one-connected combinations of networks each satisfying one of the conditions {\em (i)}-{\em (vi)}. 
\end{quotation}
A detailed and rigorous mathematical derivation and statement of the above analytical result can be found in the SI.
%
%
%

After having analytically established condition \eqref{eq: sync condition - gamma} for a variety of particular network topologies and parameters, we establish its correctness and predictive power for arbitrary networks. 
Extensive simulation studies lead to the conclusion that the proposed synchronization condition \eqref{eq: sync condition - gamma} is 
statistically correct. In order to verify this hypothesis, we conducted Monte Carlo simulation studies over a wide range of natural frequencies $\omega_{i}$, network sizes $n$, coupling weights $a_{ij}$, and different random graph models 
of varying degrees of sparsity and randomness. In total, we constructed $1.2 \cdot 10^{6}$ samples of nominal random networks, each with a connected graph $G(\mc V,\mc E, A)$ and natural frequencies $\omega$ satisfying $\| L^{\dagger} \omega \|_{\mc E,\infty} \leq \sin(\gamma)$ for some $\gamma < \pi/2$.
The detailed results can be found in the SI and allow us to establish the following probabilistic result with a confidence level of at least 99\% and accuracy of at least\,99\%:
\begin{quotation}
\noindent
{\bf Statistical result:} 
With  99.97 \% probability, for a nominal network, condition \eqref{eq: sync condition - gamma} guarantees the existence of an unique and stable solution $\theta^{*}$ with synchronized frequencies and cohesive phases $|\theta _{i}^{*} - \theta _{j}^{*}| \leq \gamma$ 
for every pair of connected oscillators $\{i,j\} \in \mc E$.
\end{quotation}
From this statistical result, we deduce that the proposed synchronization condition \eqref{eq: sync condition - gamma} holds true for {\em almost all} network topologies and parameters. Indeed, we also show the existence of possibly-thin sets of topologies and parameters for which our condition \eqref{eq: sync condition - gamma} is not sufficiently tight. We refer to the SI for an explicit family of carefully engineered and ``degenerate'' counterexamples.
Overall, our analytical and statistical results validate the correctness of the proposed condition~\eqref{eq: sync condition}. 

After having established the statistical correctness of condition \eqref{eq: sync condition - gamma}, we now investigate its predictive power for arbitrary networks.
%
Since we analytically establish that condition \eqref{eq: sync condition - gamma} is exact for sufficiently small pairwise phase cohesiveness $| \theta_{i} - \theta_{j} | \ll 1$, we now investigate the other extreme, $\max_{\{i,j\} \in \mc E} | \theta_{i} - \theta_{j} | = \pi/2$. To test the corresponding condition \eqref{eq: sync condition} in a low-dimensional parameter space, we consider a complex network of Kuramoto oscillators
\begin{equation}
	\dot \theta_{i}
	=
	\omega_{i} - K \cdot \sum\nolimits_{j=1}^{n} a_{ij} \sin(\theta_{i}-\theta_{j})
	\,, \quad i \in \until n \,,
	\label{eq: topological Kuramoto model}
\end{equation}
where all coupling weights $a_{ij}$ are either zero or one, and the coupling gain $K>0$ serves as
control parameter. If $L$ is the corresponding unweighted Laplacian matrix, then
 condition \eqref{eq: sync condition} reads as $K > \subscr{K}{critical}
\triangleq \| L^{\dagger} \omega \|_{\mc E,\infty}$.
Of course, the condition $K > \subscr{K}{critical}$ is only sufficient and
the critical coupling may be smaller than
$\subscr{K}{critical}$. 
In order to test the accuracy of the condition $K > \subscr{K}{critical}$,
we numerically found the smallest value of $K$ leading to synchrony with
phase cohesiveness $\pi/2$.  

Figure \ref{Fig: Statistical studies} reports
our findings for various network sizes, connected random graph models, and
sample distributions of the natural frequencies. We refer to the SI for the detailed simulation setup.
%
%
%
First, notice from Subfigures (a),(b),(d), and (e) that condition \eqref{eq: sync condition} is extremely accurate for a sparse graph, that is, for small
$p$ and $n$, as expected from our analytical results. 
Second, for a dense graph with $p\approx 1$, Subfigures (a),(b),(d), and (e) confirm the results known for classic Kuramoto oscillators \citepri{SI-FD-FB:10w}: for a bipolar distribution condition \eqref{eq: sync condition} is exact, and for a uniform distribution a small critical coupling is obtained. 
Third, Subfigures (c) and (d) show that condition \eqref{eq: sync condition} is scale-free for a Watts-Strogatz small world network, that is, it has almost constant accuracy for various values of $n$ and $p$.
Fourth and finally, observe that condition \eqref{eq: sync condition} is always within a constant factor of the exact critical coupling, whereas other proposed conditions \citepri{SI-FFW-SK:80,SI-FD-FB:09z,SI-AJ-NM-MB:04,SI-FW-SK:82,SI-AA-ADG-JK-YM-CZ:08,SI-SB-VL-YM-MC-DUH:06} on the nodal degree or on the algebraic connectivity scale poorly with respect to network size $n$. 


\begin{figure}
  \centering{
    \includegraphics[width=0.99\columnwidth]{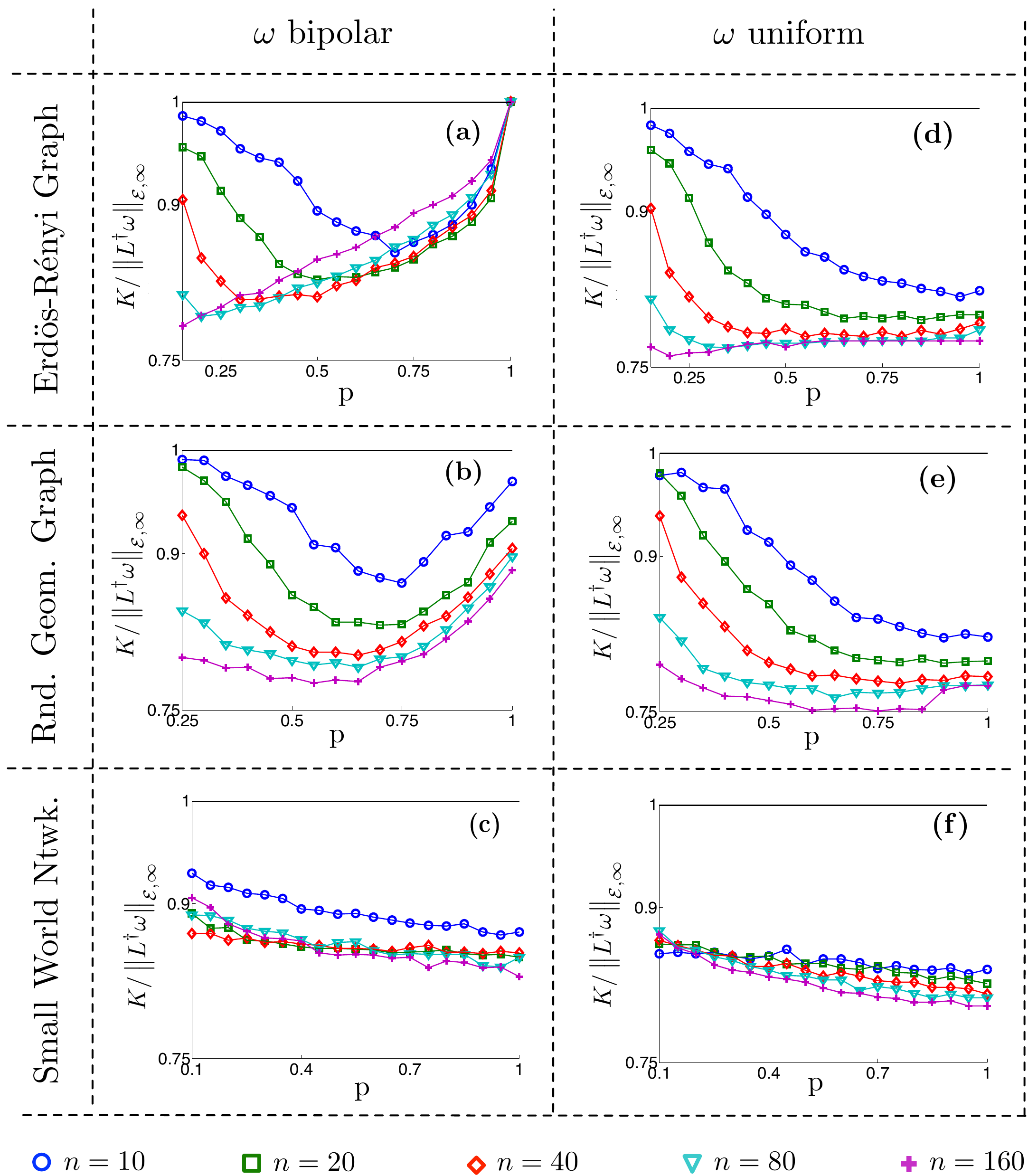}
    \caption{
Numerical evaluation of the exact critical coupling $K$ in a complex Kuramoto oscillator network. The subfigures show $K$ normalized by $\| L^{\dagger} \omega \|_{\mc E,\infty}$ for an Erd\"os-R\'enyi graph with probability $p$ of connecting two nodes, for a random geometric graph with connectivity radius $p$, and for a Watts-Strogatz small world network with rewiring probability $p$. 
	Each data point is the mean over 100 samples of the respective random graph model, for values of $\omega_{i}$ sampled from a bipolar or a uniform distribution supported on ${[-1,1]}$, and for the network sizes $n \in \{10,20,40,80,160\}$.
	}
	\label{Fig: Statistical studies}
	}
\end{figure}

\section*{Applications in Power Networks}

We envision that condition \eqref{eq: sync condition - gamma} can be applied to quickly assess synchronization and robustness in power networks under volatile operating conditions. 
%
Since real-world power networks are carefully engineered systems with particular network topologies and parameters, 
we do not extrapolate the statistical results from the previous section to power grids. Rather, we consider ten widely-established IEEE power network test cases provided by \citepri{SI-RDZ-CEM-DG:11,SI-CG-PW-PA-RA-MB-RB-QC-CF-SH-SK-WL-RM-DP-NR-DR-AS-MS-CS:99}. 

Under nominal operating conditions, the power generation is optimized to
meet the forecast demand, while obeying the AC power flow laws and
respecting the thermal limits of each transmission line. Thermal limits
constraints are precisely equivalent to phase cohesiveness requirements.
In order to test the synchronization condition \eqref{eq: sync condition - gamma} in a volatile smart grid scenario, we make the following changes to the nominal network: 1) We assume fluctuating demand and randomize 50\% of all loads to deviate from the forecasted loads. 2) We assume that the grid is penetrated by renewables with severely fluctuating power outputs, for example, wind or solar farms, and we randomize 33\% of all generating units to deviate from the nominally scheduled generation. 3) Following the paradigm of {\em smart operation of smart grids} \citepri{SI-PPV-FFW-JWB:11}, the fluctuations can be mitigated by fast-ramping generation, such as fast-response energy storage including batteries and flywheels, and controllable loads, such as large-scale server farms or fleets of plug-in hybrid electrical vehicles. Here, we assume that the grid is equipped with 10\% fast-ramping generation and 10\% controllable loads, and the power imbalance (caused by fluctuating demand and generation) is uniformly dispatched among these adjustable power sources.
For each of the ten IEEE test cases, we construct 1000 random realizations of the scenario 1), 2), and 3) described above, we numerically check for the existence of a synchronous solution, and we compare the numerical solution with the results predicted by our synchronization condition \eqref{eq: sync condition - gamma}. Our findings are reported in Table \ref{Table: sync condition for volatile power test cases}, and a detailed description of the simulation setup can be found in the SI.
It can be observed that condition \eqref{eq: sync condition - gamma} predicts the correct phase cohesiveness $| \theta_{i} - \theta_{j}|$ along all transmission lines $\{i,j\} \in \mc E$ with extremely high accuracy even for large-scale networks featuring 2383 nodes. 

\begin{figure}
	\centering{
	\includegraphics[width=.75\columnwidth]{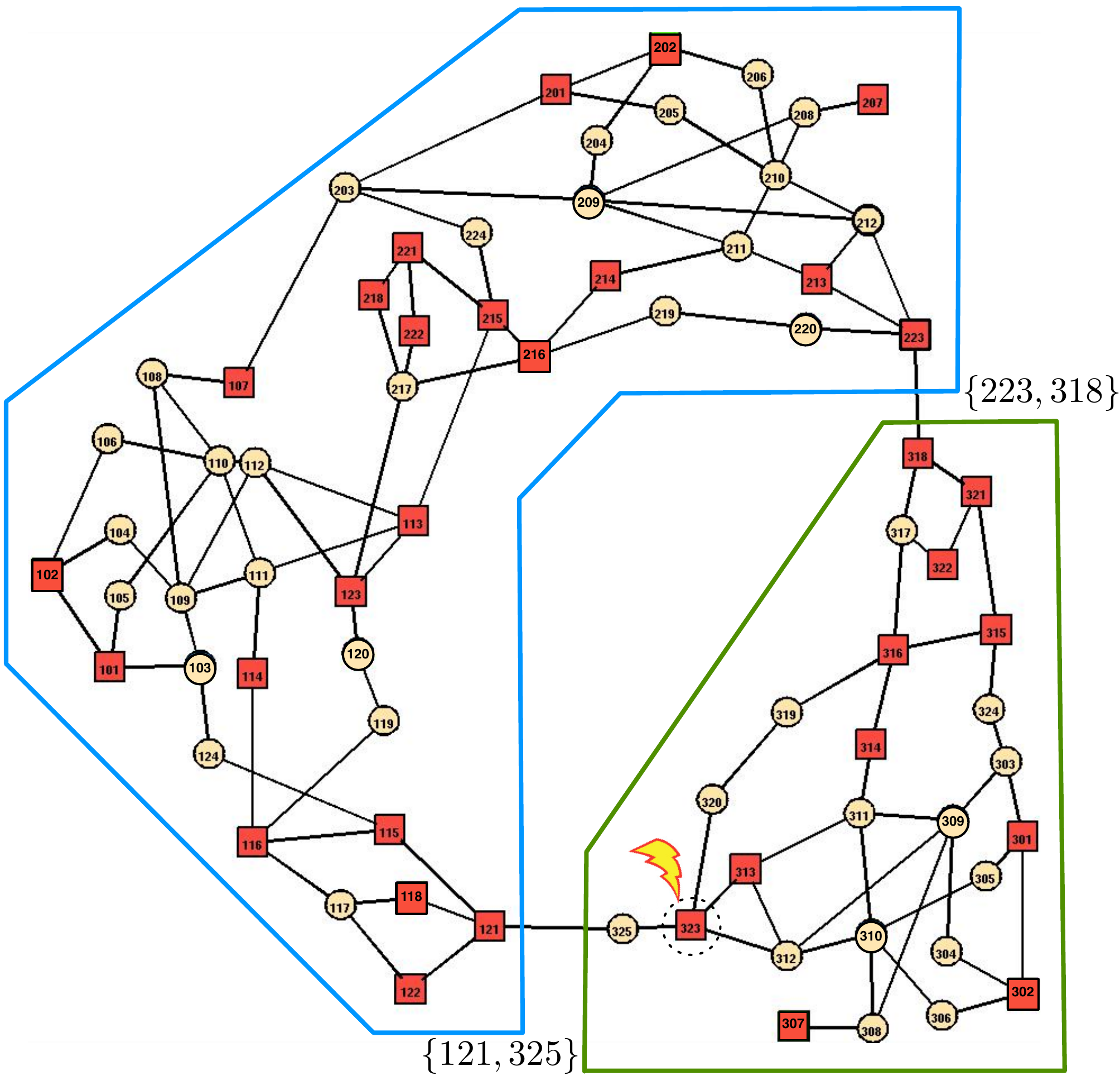}
	\caption{Illustration of contingencies the RTS 96 power network. Here, square nodes are generators and round nodes are
           loads, large amounts of power are exported from the Northwestern area to the Southeastern area, and generator 323 is tripped.}
	\label{Fig: RTS 96}
	}
\end{figure}
As a final test, we validate the synchronization condition \eqref{eq: sync
  condition - gamma} in a stressed power grid case study. We consider the
{\em IEEE Reliability Test System 96} (RTS 96)
\citepri{SI-CG-PW-PA-RA-MB-RB-QC-CF-SH-SK-WL-RM-DP-NR-DR-AS-MS-CS:99} illustrated
in Figure~\ref{Fig: RTS 96}.
We assume the following two contingencies have taken place and we
characterize the remaining safety margin. First, we assume generator 323 is
disconnected, possibly due to maintenance or failure events. Second, we
consider the following imbalanced power dispatch situation: the power demand at 
each load in the Southeastern area deviates from the nominally forecasted demand 
by a uniform and positive amount, and the resulting power deficiency is compensated by
uniformly increasing the generation in the Northwestern area.  This
imbalance can arise, for example, due to a shortfall in predicted load and renewable
energy generation.  Correspondingly, power is exported from the
Northwestern to the Southeastern area via the transmission
lines $\{121, 325\}$ and $\{223, 318\}$.
At a nominal operating condition, the RTS 96 power network is sufficiently robust to tolerate each single one
of these two contingencies, but the safety margin is now minimal.  When
both contingencies are combined, then our synchronization condition
\eqref{eq: sync condition - gamma} predicts that the thermal limit of the
transmission line $\{121, 325\}$ is reached at an additional loading of
22.20\%.
Indeed, the dynamic simulation scenario shown in Figure \ref{Fig: RTS 96 -
  dynamic simulation} validates the accuracy of this prediction. It can be
observed, that synchronization is lost for an additional loading of
22.33\%, and the areas separate via the transmission line $\{121,
325\}$. This separation triggers a cascade of events, such as the outage of
the transmission line $\{223, 318\}$, and the power network is en route to
a blackout.
We remark that, if generator 323 is not disconnected and there are no
thermal limit constraints, then, by increasing the loading, we observe the
classic loss of synchrony through a saddle-node bifurcation. Also this
bifurcation can be predicted accurately by our results, see the SI for a detailed
description.
%

\begin{figure}
	\centering{
	\includegraphics[width=0.99\columnwidth]{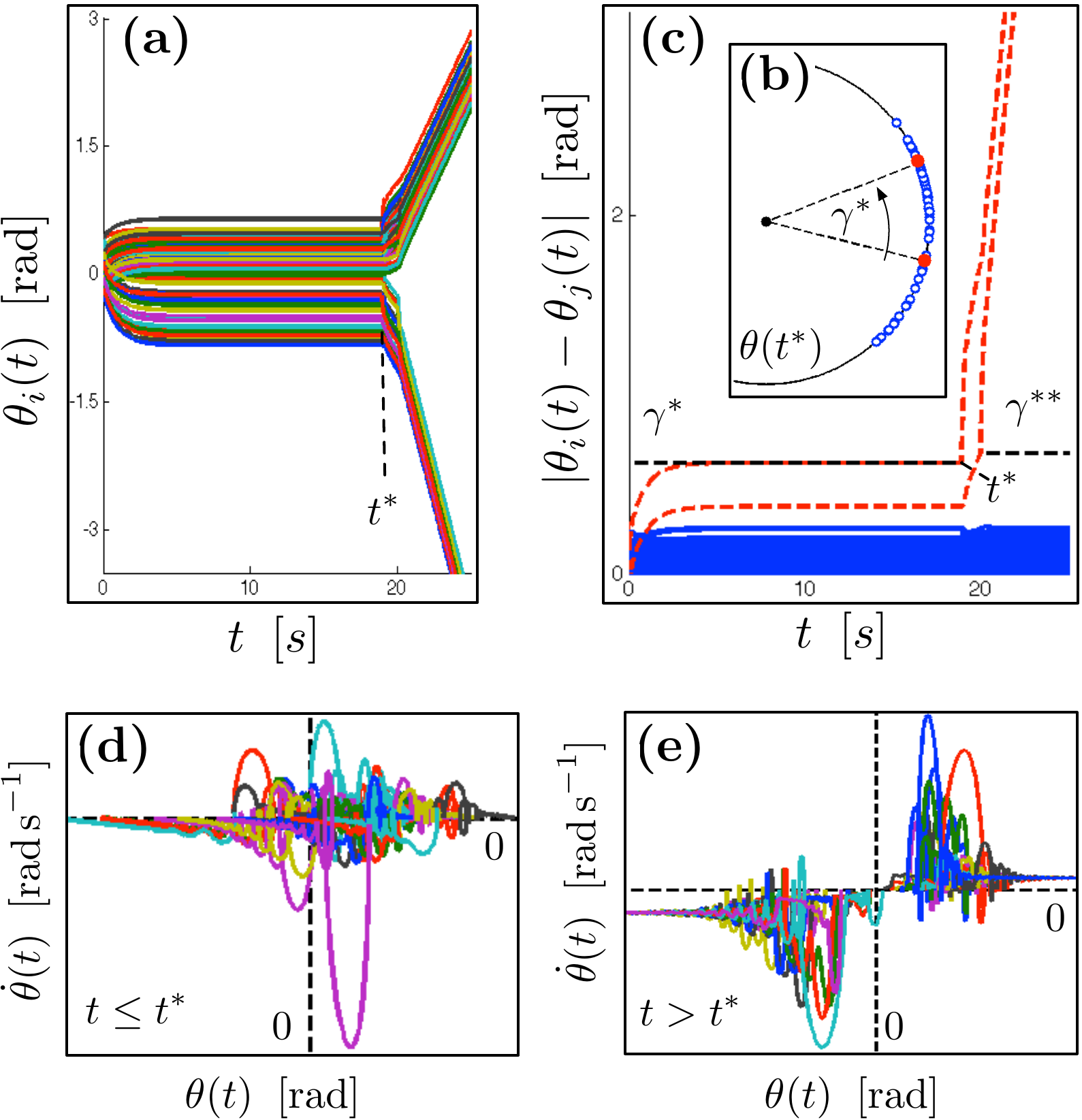}
	\caption{The RTS 96 dynamics for a continuous load increase from 22.19\% to 22.24\%. Subfigure (a) shows the angles $\theta(t)$ which loose synchrony at $t^{*} = 18.94\,$s, when the thermal limit $\gamma^{*}= 0.1977 \textup{\,rad}$ of the transmission line $\{121, 325\}$ is reached. Subfigure (b) shows the angles $\theta(t)$ at $t = t^{*}$. Subfigure (c) depicts the angular distances and the thermal limits $\gamma^{*}$ and $\gamma^{**}$ of the lines $\{121, 325\}$ and $\{223, 318\}$. Subfigures (d) and (e) show the generator phase space $\bigl(\theta(t),\dot\theta(t)\bigr)$ before and after $t^{*}$, where the loss of a common synchronization frequency can be observed.}
	\label{Fig: RTS 96 - dynamic simulation}
	}
\end{figure}

In summary, the results in this section confirm the validity, the applicability, and the accuracy of the synchronization condition \eqref{eq: sync condition - gamma} in  complex power network scenarios.

\section*{Discussion and Conclusions}
In this article we studied the synchronization phenomenon for broad class
of coupled oscillator models proposed in the scientific literature. We
proposed a surprisingly simple condition that accurately predicts
synchronization as a function of the parameters and the topology of the
underlying network. Our result, with its physical and graph theoretical
interpretations, significantly improves upon the existing test in the
literature on synchronization.  The correctness of our synchronization
condition is established analytically for various interesting network
topologies and via Monte Carlo simulations for a broad range of generic
networks.  We validated our theoretical results for complex Kuramoto
oscillator networks as well as in smart grid applications.

Our results equally answer as many questions as they pose. Among the
important theoretical problems to be addressed is a characterization of the
set of all network topologies and parameters for which our proposed
synchronization condition $\|L^{\dagger} \omega \|_{\mc E, \infty} < 1$ is
not sufficiently tight. We conjecture that this set is ``thin'' in an
appropriate parameter space. Our results suggest that an exact condition
for synchronization of any arbitrary network is of the form $\|L^{\dagger}
\omega \|_{\mc E, \infty} < c$, and we conjecture that the constant $c$ is
always strictly positive, upper-bounded, and close to one. 
Yet another important question not addressed in the
present article concerns the region of attraction of a synchronized
solution. We conjecture that the latter depends on the gap in the presented
synchronization condition.
On the application side, we envision that our synchronization conditions
enable emerging smart grid applications, such as power flow optimization
subject to stability constraints, distance to failure metric, and the
design of control strategies to avoid cascading failures.

\begin{table}[h]
\caption{Evaluation of condition \eqref{eq: sync condition - gamma} for ten IEEE test cases under volatile operating conditions.}
\label{Table: sync condition for volatile power test cases}
\end{table}
\vspace{-2cm}

\begin{table}[h]
\centering
\caption{Evaluation of condition \eqref{eq: sync condition - gamma} for ten IEEE test cases under volatile operating conditions.}
     \label{Table: sync condition for volatile power test cases}
     \caption{ Evaluation of condition \eqref{eq: sync condition - gamma} for ten IEEE test cases under volatile operating conditions.}
	\small
  \begin{tabular}{ | l | l | l | l | r}
    \hline & & &\\
     \!\!Randomized test\!\!  &  
     \!\!\tablenote{Correctness: $ \Big. \| L^{\dagger} \omega \|_{\mc E,\infty} \!\leq\! \sin(\gamma)$ $\implies$ $\max\nolimits_{\{i,j\} \in \mc E} | \theta_{i}^{*} - \theta_{j}^{*}| \leq \gamma$}{Correctness:}\!\! &  
     \!\!\tablenote{Accuracy: $\max\nolimits_{\{i,j\} \in \mc E} | \theta_{i}^{*} - \theta_{j}^{*}| -\arcsin(\| L^{\dagger} \omega \|_{\mc E,\infty})$ }{Accuracy:}\!\! & 
     \!\!\tablenote{Phase cohesiveness: $\max\nolimits_{\{i,j\} \in \mc E} | \theta_{i}^{*} - \theta_{j}^{*}|$}{Cohesive}\!\! \\
     \!\!case (1000 instances):$\Big.$\!\!  & & & \!\!{phases:}\!\!
    \\\hline\hline
     \!\!Chow 9 bus system\!\!  &  \!\tt{always true}\! &   \!\!$4.1218 \cdot 10^{-5}$\!\!  &  $0.12889$ 
    \\\hline    
     \!\!IEEE 14 bus system\!\! &  \!\tt{always true}\! &  \!\!$2.7995 \cdot 10^{-4}$\!\!   &  $0.16622$ 
    \\\hline
     \!\!IEEE RTS 24\!\! &  \!\tt{always true}\! &  \!\!$1.7089 \cdot 10^{-3}$\!\!   &  $0.22309$ 
    \\\hline
     \!\!IEEE 30 bus system\!\! &  \!\tt{always true}\! &  \!\!$2.6140 \cdot 10^{-4}$\!\!  &  $0.1643$ 
    \\\hline
     \!\!New England 39\!\! &  \!\tt{always true}\! &   \!\!$6.6355 \cdot 10^{-5}$\!\! &  $0.16821$ 
    \\\hline
     \!\!IEEE 57 bus system\!\! &  \!\tt{always true}\! &   \!\!$2.0630 \cdot 10^{-2}$\!\!   &  $0.20295$ 
    \\\hline
     \!\!IEEE RTS 96\!\! &  \!\tt{always true}\! &  \!\!$2.6076 \cdot 10^{-3}$\!\!  &   $0.24593$ 
    \\\hline
     \!\!IEEE 118 bus system\!\! &  \!\tt{always true}\! &  \!\!$5.9959 \cdot 10^{-4}$\!\!  &   $0.23524$ 
    \\\hline
     \!\!IEEE 300 bus system\!\! &  \!\tt{always true}\! &  \!\!$5.2618 \cdot 10^{-4}$\!\!  &   $0.43204$ 
    \\\hline
    \!\!Polish 2383 bus\!\! &  \!\tt{always true}\! &  \!\!$4.2183 \cdot 10^{-3}$\!\!  &   $0.25144$ 
    \\ \!\!\!system (winter 99)\!\! & & &
    \\\hline
  \end{tabular}
         \tablenote{}{The accuracy and phase cohesiveness results in the third and fourth column are given in the unit $[\textup{rad}]$, and they are averaged over 1000 instances of randomized load and generation.}
\end{table}

\begin{acknowledgments}
This material is based in part upon work supported by NSF grants IIS-0904501 and CPS-1135819. Research at LANL was carried out under the auspices of the National Nuclear Security Administration of the U.S. Department of Energy at Los Alamos National Laboratory under Contract No. DE C52-06NA25396.
\end{acknowledgments}





%


\end{article}








\begin{article}
\newpage
\setcounter{equation}{0}
{\bf \Large Supplementary Information}\\

\section{Introduction}
\label{Section: Introduction}

This supplementary information is organized as follows.

The section {\em Mathematical Models and Synchronization Notions} provides a description of the considered coupled oscillator model including a detailed modeling of a mechanical analog and a few power network models.
Furthermore, we state our definition of synchronization and compare various synchronization conditions proposed for oscillator networks.

The section {\em Mathematical Analysis of Synchronization} provides a rigorous mathematical analysis of synchronization, which leads to the novel synchronization conditions proposed in the main article. Throughout our analysis we provide various examples illustrating certain theoretical concepts and results, and we also compare our results to existing results in the synchronization and power networks literature.

The section {\em Statistical Synchronization Assessment} provides a detailed account of our Monte Carlo simulation studies and the complex Kuramoto network studies. Throughout this section, we also recall the basics of probability estimation by Monte Carlo methods that allow us to establish a statistical synchronization result in a mathematically rigorous way.

Finally, the section {\em Synchronization Assessment for Power Networks} describes the detailed simulation setup for the randomized IEEE test systems, it provides the simulation data used for the dynamic IEEE RTS 96 power network simulations, and it illustrates a dynamic bifurcation scenario in the IEEE RTS 96 power network.

The remainder of this section introduces some notation and recalls some preliminaries.

\subsection{Preliminaries and Notation}

{\it Vectors and functions:}  
Let $\fvec 1_{n}$ and $\fvec 0_{n}$ be the $n$-dimensional vector of unit and zero entries, and let $\fvec 1_n^\perp$ be the orthogonal complement of $\fvec 1_{n}$ in $\mbb R^{n}$, that is, $\fvec 1_n^\perp \triangleq \{ x \in \mbb R^{n} :\, x \perp \fvec 1_{n}\}$. Let $e^{n}_{i}$ be $i$th canonical basis vector of $\mathbb R^{n}$, that is, the $i$th entry of $e^{n}_{i}$ is 1 and all other entries are zero. 
Given an $n$-tuple $(x_{1},\dots,x_{n})$, let $x \in \real^{n}$ be the associated vector. For an ordered index set $\mc I$ of cardinality $|\mc I|$ and an one-dimensional array $\{x_{i}\}_{i \in \mc I}$, we define $\diag(\{c_{i}\}_{i \in \mc I}) \in \mbb R^{|\mc I| \times |\mc I|}$ to be the associated diagonal matrix. 
%
For $x \in \mathbb R^{n}$, define the vector-valued functions $\sinbf(x) = (\sin(x_{1}),\dots,\sin(x_{n}))$ and $\arcsinbf(x)=(\arcsin(x_{1}),\dots,\arcsin(x_{n}))$, where the $\arcsin$ function is defined for the branch ${[-\pi/2,\pi/2]}$.
For a set $\mc X \subset \real^{n}$ and a matrix $A \in \real^{m \times n}$, let $A \mc X  = \{ y \in \mathbb R^{m}:\, y = Ax \,,\, x \in \mc X\}$.

{\it Geometry on $n$-torus:} The set $\mbb S^{1}$ denotes the {\em unit circle},
an {\it angle} is a point $\theta \in \mbb S^{1}$, and an {\it arc} is a
connected subset of $\mbb S^{1}$. 
The {\em geodesic distance} between two angles $\theta_{1}$, $\theta_{2} \in \mbb S^{1}$ is the minimum of the counter-clockwise and the clockwise arc length connecting $\theta_{1}$ and $\theta_{2}$.
With slight abuse of notation, let $|\theta_{1}-\theta_{2}|$ denote the {\it geodesic distance} between two
angles $\theta_{1},\theta_{2} \in \mbb S^{1}$. 
Finally, the {\em $n$-torus} is the product set $\mbb T^{n} = \mbb S^{1} \times  \dots \times \mbb S^{1}$ is the
direct sum of $n$ unit circles.  

{\it Algebraic graph theory:}  Given an undirected, connected, and weighted graph $G(\mc V,\mc E,A)$ induced by the symmetric, irreducible, and nonnegative {\em adjacency matrix} $A \in \mathbb R^{n \times n}$, the {\em Laplacian matrix} $L \in \real^{n \times n}$ is defined by $L  = \diag(\{\sum_{j=1}^{n} a_{ij} \}_{i =1}^{n}) - A$. 
If a number $\ell \in \until {|\mc E|}$ and an arbitrary direction is
assigned to each edge $\{i,j\} \in \mc E$, the (oriented) {\em incidence
  matrix} $B\in \real^{n \times |\mc E|}$ is defined component-wise as
$B_{k\ell} = 1$ if node $k$ is the sink node of edge ${\ell}$ and as
$B_{k\ell} = -1$ if node $k$ is the source node of edge ${\ell}$; all other
elements are zero.  For $x \in \real^{n}$, the vector $B^{T}x$ has
components $x_{i} - x_{j}$ for any oriented edge from $j$ to $i$, that is, $B^{T}$ maps node variables $x_{i}$, $x_{j}$ to incremental edge variables $x_{i}-x_{j}$. 
If $\diag(\{a_{ij}\}_{\{i,j\} \in \mc E})$ is the diagonal matrix of nonzero edge
weights, then $L = B \diag(\{a_{ij} \}_{\{i,j\} \in \mc E}) B^{T}$.
For a vector $x \in \mathbb R^{n}$, the incremental norm $\|x\|_{\mc E,\infty} \triangleq \max_{\{i,j\} \in \mc E}$ used in the main article, can be expressed via the incidence matrix $B$ as $\|x\|_{\mc E,\infty} = \| B^{T} x \|_{\infty}$.
If the graph is connected, then $\Ker(B^{T}) = \Ker(L) = \mathrm{span}(\fvec 1_{n})$, all $n-1$ remaining eigenvalues of $L$ are real and strictly positive, and the second-smallest eigenvalue $\lambda_{2}(L)$ is called the {\it algebraic connectivity}. The orthogonal vector spaces $\Ker(B)$ and $\Ker(B)^{\perp} = \Img(B^{T})$ are spanned by vectors associated to cycles and cut-sets in the graph 
, see for example \citesec[Section 4]{NB:94} or \citesec{NB:97}. 
In the following, we refer to $\Ker(B)$ and $\Img(B^{T})$ as the {\em cycle space} and the {\em cut-set space}, respectively. 

{\it Laplacian inverses:}
Since the Laplacian matrix $L$ is singular, we will frequently use its {\em
  Moore-Penrose pseudo inverse} $L^{\dagger}$. If $U \in \real^{n \times
  n}$ is an orthonormal matrix of eigenvectors of $L$, the singular value decomposition of $L$ is
  $L = U \diag(\{0,\lambda_{2},\dots,\lambda_{n}\}) U^{T}$, and its Moore-Penrose pseudo inverse $L^{\dagger}$ is given by 
  $ L^{\dagger} = U \diag(\{0,1/\lambda_{2},\dots,1/\lambda_{n}\}) U^{T} $.
  We will frequently use the identity $L \cdot L^{\dagger} = L^{\dagger} \cdot L = I_{n} - \frac{1}{n}
\fvec 1_{n \times n}$, which follows directly from the singular value
decomposition. We also define the {\em effective resistance} between nodes
$i$ and $j$ by $R_{ij} = L^{\dagger}_{ii} + L^{\dagger}_{jj} - 2
L^{\dagger}_{ij}$. We refer to~\citesec{FD-FB:11d} for further information on Laplacian inverses and on the resistance~distance.

\section{Mathematical Models and Synchronization Notions}
\label{Section: Mathematical Models and Synchronization Notions}

In this section we introduce the mathematical model of coupled phase oscillators considered in this article, we present some synchronization notions, and give a detailed account of the literature on synchronization of coupled phase oscillators.

\subsection{General Coupled Oscillator Model} 

Consider a weighted, undirected, and connected graph $G = (\mc V,\mc E,A)$ with $n$ nodes $\mc V = \until n$, partitioned node set $\mc V = \mc V_{1} \cup \mc V_{2}$ and edge set  $\mc E$ induced by the adjacency matrix $A \in \mathbb R^{n \times n}$. We assume that the graph $G$ has no self-loops $\{i,i\}$, that is, $a_{ii} = 0$ for all $i \in \mc V$. Associated to this graph, consider the following model of $|\mc V_{1}| \geq 0$ second-order Newtonian and $|\mc V_{2}|  \geq 0$ first-order kinematic phase oscillators
\begin{align}
	\begin{split}
	M_{i} \ddot \theta_{i} + D_{i} \dot \theta_{i}
	&=
	\omega_{i} - \sum\nolimits_{j=1}^{n} a_{ij} \sin(\theta_{i}-\theta_{j})
	\,,\qquad i \in \mc V_{1},
	\\
	D_{i} \dot \theta_{i}
	&=
	\omega_{i} - \sum\nolimits_{j=1}^{n} a_{ij} \sin(\theta_{i}-\theta_{j})
	\,,\qquad i \in \mc V_{2},
	\end{split}
	\label{eq-SI: coupled oscillator model}
\end{align}
where $\theta_{i} \in \mathbb S^{1}$ and $\dot \theta_{i} \in \mathbb R^{1}$ are the phase and frequency of oscillator $i \in \mc V$, $\omega_{i} \in \mathbb R^{1}$ and $D_{i} > 0$ are the natural frequency and damping coefficient of oscillator $i \in \mc V$, and $M_{i} > 0$ is inertial constant of oscillator $i \in \mc V_{1}$. 
The coupled oscillator model \eqref{eq-SI: coupled oscillator model} evolves on $\mathbb T^{n} \times \mathbb R^{|\mc V_{1}|}$, and features an important symmetry, namely the rotational invariance of the angular variable~$\theta$. 
The interesting dynamics of the coupled oscillator model \eqref{eq-SI: coupled oscillator model} arises from a competition between each oscillator's tendency to align with its natural frequency $\omega_{i}$ and the synchronization-enforcing coupling $a_{ij} \sin(\theta_{i} - \theta_{j})$ with its neighbors.

As discussed in the main article, the coupled oscillator model \eqref{eq-SI: coupled oscillator model} unifies various models proposed in the literature. For example, for  the parameters $\mc V_{1} = \emptyset$ and $D_{i} = 1$ for all $i \in\mc V_{2}$, it reduces to the celebrated {\em Kuramoto model} \citesec{YK:75,YK:84}
\begin{equation}
	\dot \theta_{i}
	=
	\omega_{i} - \sum\nolimits_{j=1}^{n} a_{ij} \sin(\theta_{i}-\theta_{j})
	\,,\;\;\; i \in \until n 
	\label{eq-SI: Kuramoto model}
	\,.
\end{equation}
We refer to the review articles \citesec{SHS:00,JAA-LLB-CJPV-FR-RS:05,ATW:01,FD-FB:10w,FD-FB:12i} for various theoretic results on the Kuramoto model \eqref{eq-SI: Kuramoto model} and further synchronization applications in natural sciences, technology, and social networks.
Here, we present a detailed modeling of the spring oscillator network used as a mechanical analog in the main article, and we present a few power network models, which can be described by the coupled oscillator model \eqref{eq-SI: coupled oscillator model}.

\subsection{Mechanical Spring Network}

Consider the spring network illustrated in Figure \ref{Fig: Mechanical Analog}
consisting of a group of $n$ particles constrained to rotate around a
circle with unit radius. For simplicity, we assume that the particles are allowed to move
freely on the circle and exchange their order without collisions. 

Each
particle is characterized by its phase angle $\theta_{i} \in \mathbb S^{1}$
and frequency $\dot \theta_{i} \in \real$, and its inertial and damping
coefficients are $M_{i}>0$ and $D_{i}>0$.
The external forces and torques acting on each particle are (i) a viscous
damping force $D_{i} \dot \theta_{i}$ opposing the direction of motion,
(ii) a non-conservative force $\omega_{i} \in \real$ along the direction of
motion depicting a preferred natural rotation frequency, and (iii) an
elastic restoring torque between interacting particles $i$ and $j$ coupled
by an ideal elastic spring with stiffness $a_{ij} > 0$ and zero rest
length. The topology of the spring network is described by the weighted, undirected, and connected graph $G = (\mc V,\mc E,A)$.

To compute the elastic torque between the particles, we parametrize the
position of each particle $i$ by the unit vector $p_{i} = \left[
  \cos(\theta_{i}) \,,\, \sin(\theta_{i}) \right]^{T} \in \mathbb S^{1}
\subset \real^{2}$. The elastic Hookean energy stored in the springs is the
function $E:\, \torus^{n} \to \real$ given up to an additive constant by
\begin{align*}
 E(\theta) 
 &= \sum\nolimits_{\{i,j\} \in \mc E} \frac{a_{ij}}{2}  \| p_{i} - p_{j} \|_{2}^{2} \\
 &= \sum\nolimits_{\{i,j\} \in \mc E} a_{ij}  \bigl( 1 -\cos(\theta_{i})\cos(\theta_{j}) - \sin(\theta_{i})\sin(\theta_{j}) \bigr) \\
 &= \sum\nolimits_{\{i,j\} \in \mc E} a_{ij}  \bigl( 1 - \cos(\theta_{i} - \theta_{j})\bigr)
 \,,
\end{align*}
where we employed the trigonometric identity $\cos(\alpha - \beta) = \cos
\alpha \cos \beta + \sin \alpha \sin \beta$ in the last equality. Hence, we
obtain the restoring torque acting on particle $i$ as
\begin{equation*}
 T_{i}(\theta) = - \frac{\partial}{\partial\theta_{i}} \, E(\theta) = - \sum\nolimits_{j=1}^{n} a_{ij} \sin(\theta_{i} - \theta_{j})
 \,.
\end{equation*}
Therefore, the network of spring-interconnected particles depicted in
Figure \ref{Fig: Mechanical Analog} obeys the dynamics
\begin{equation}
 M_i \ddot{\theta}_i + D_i \dot{\theta}_i = \omega_{i} - \sum\nolimits_{j=1}^{n} a_{ij}\sin(\theta_i-\theta_j)
\,,\;\; i \in \until n.
 \label{eq-SI: spring network}
\end{equation}
In conclusion, the spring network in Figure~\ref{Fig: Mechanical Analog} is a mechanical analog of the 
 coupled oscillator model \eqref{eq-SI: coupled oscillator model} with	~$\mc V_{2} = \emptyset$.
\begin{figure}[htbp]
	\centering{
	\includegraphics[width=0.4\columnwidth]{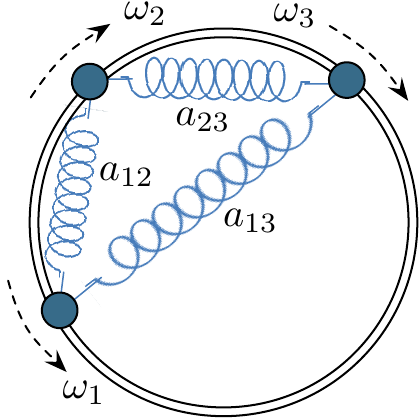}
	\caption{Mechanical analog of the coupled oscillator model \eqref{eq-SI: coupled oscillator model}.}
	\label{Fig: Mechanical Analog}
	}
\end{figure}

\subsection{Power Network Model}

The coupled oscillator model \eqref{eq-SI: coupled oscillator model} includes also a variety of power network models.
We briefly present different power network models compatible with the coupled oscillator model \eqref{eq-SI: coupled oscillator model} and refer to \citesec[Chapter 7]{PWS-MAP:98} for a detailed derivation from a higher order first principle model. 

Consider a connected power network with generators $\mc V_{1}$ and load
buses $\mc V_{2}$. The network is described by the symmetric nodal
admittance matrix $Y \in \mbb C^{n \times n}$ (augmented with the generator
transient reactances). If the network is lossless and the voltage levels
$|V_{i}|$ at all nodes $i \in \mc V_{1} \cup \mc V_{2}$ are constant, then
the {\it maximum real power transfer} between any two nodes $i,j \in \mc
V_{1} \cup \mc V_{2}$ is $a_{ij} = |V_{i}| \cdot |V_{j}| \cdot
\Im(Y_{ij})$, where $\Im(Y_{ij})$ denotes the susceptance of the transmission line $\{i,j\} \in \mc E$. With this notation the
swing dynamics of generator $i$ are given by
\begin{equation}
	M_{i} \ddot \theta_{i} + D_{i} \dot \theta_{i}
	\!=\! 
	P_{\textup{m},i} - \sum\nolimits_{j=1}^{n} a_{ij} \sin(\theta_{i} - \theta_{j})
	\,,\;\;\;  i \in \mc V_{1} ,
	\label{eq-SI: generator dynamics}
\end{equation}
where $\theta_{i} \in \mbb S^{1}$ and $\dot \theta_{i} \in \real^{1}$ are the generator rotor angle and frequency, $\theta_{j} \in \mbb S^{1}$ for $j \in \mc V_{2}$ are the voltage phase angles at the load buses, and $P_{\textup{m},i} > 0$, $M_{i} > 0$, and $D_{i} > 0$ are the mechanical power input from the prime mover, the generator inertia constant, and the damping coefficient. 

For the load buses $\mc V_{2}$, we consider the following three load models illustrated in Figure \ref{Fig: load models}.

{\em 1) PV buses with frequency-dependent loads:} All load buses are $PV$ {\em buses}, that is, the active power demand $P_{\textup{l},i}$ and the voltage magnitude $|V_{i}|$ are specified for each bus. The real power drawn by load $i$ consists of a constant term $P_{\textup{l},i} > 0$
and a frequency dependent term $D_{i} \dot \theta_{i}$ with $D_{i}>0$, as illustrated in Figure \ref{Fig: load models}(a). The
resulting real power balance equation is
\begin{equation}
	D_{i} \dot \theta_{i} + P_{\textup{l},i}
	=
	- \sum\nolimits_{j=1}^{n} a_{ij} \sin(\theta_{i} - \theta_{j})
	\,,\;\;\; i \in \mc V_{2} \,.
	\label{eq-SI: bus dynamics}
\end{equation}
The dynamics \eqref{eq-SI: generator dynamics}-\eqref{eq-SI: bus dynamics} are known as {\em structure-preserving power network model} \citesec{ARB-DJH:81}, and equal the coupled oscillator model \eqref{eq-SI: coupled oscillator model} for $\omega_{i} = P_{\textup{m},i}$, $i \in \mc V_{1}$, and $\omega_{i} = -P_{\textup{l},i}$, $i \in \mc V_{2}$.

{\em 2) PV buses with constant power loads:}
All load buses are $PV$ {\em buses}, each load features a constant real power demand $P_{\textup{l},i} > 0$, and the load damping in \eqref{eq-SI: bus dynamics} is neglected, that is, $D_{i}=0$ in equation \eqref{eq-SI: bus dynamics}. The corresponding circuit-theoretic model is shown in Figure \ref{Fig: load models}(b). If the angular distances $|\theta_{i}(t) - \theta_{j}(t)| < \pi/2$ are bounded for each transmission line $\{i,j\} \in \mc E$ (this condition will be precisely established in the next section), then the resulting differential-algebraic system has the same local stability properties as the dynamics \eqref{eq-SI: generator dynamics}-\eqref{eq-SI: bus dynamics}, see \citesec{SS-PV:80}. Hence, all of our results apply locally also to the structure-preserving power network model \eqref{eq-SI: generator dynamics}-\eqref{eq-SI: bus dynamics} with zero load damping $D_{i}=0$ for $i \in \mc V_{2}$.

{\em 3) Constant current and constant admittance loads:}
If each load $i \in \mc V_{2}$ is modeled as a constant current demand $I_{i}$ and an (inductive) admittance $Y_{i,\textup{shunt}}$ to ground as illustrated in Figure \ref{Fig: load models}(c), then the linear current-balance equations are $I = Y V$, where $I \in \mathbb C^{n}$ and $V \in \mathbb C^{n}$ are the vectors of nodal current injections and voltages. After elimination of the bus variables $V_{i}$, $i \in \mc V_{2}$, through Kron reduction \citesec{FD-FB:11d}, the resulting dynamics assume the form \eqref{eq-SI: spring network} known as the (lossless) {\em network-reduced power system model} \citesec{HDC-CCC-GC:95,FD-FB:09z}. We refer to \citesec{PWS-MAP:98,FD-FB:11d} for a detailed derivation of the network-reduced model.

\begin{figure}[htbp]
	\centering{
	\includegraphics[width = 0.99\columnwidth]{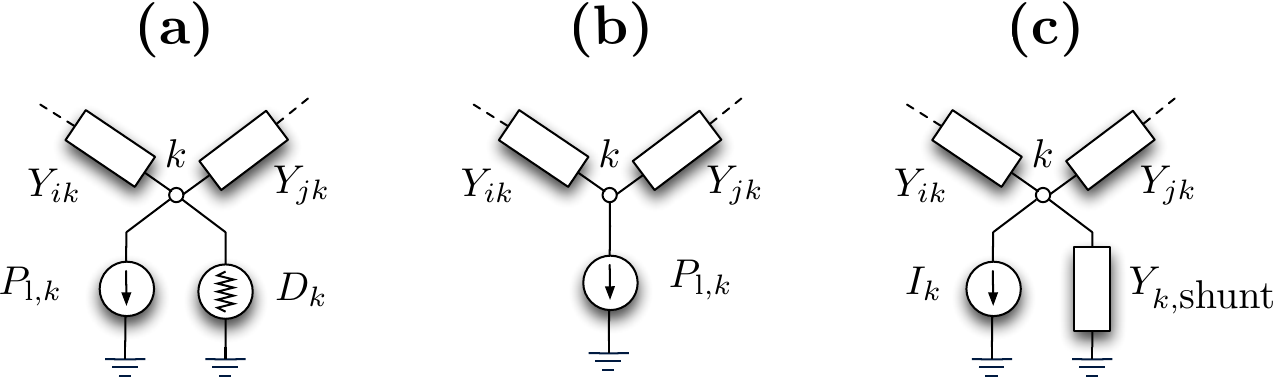}
	\caption{Equivalent circuits of the frequency-dependent load model (a), the constant power load model (b), and the constant current and admittance load model~(c).}
	\label{Fig: load models}
	}
\end{figure}

To conclude this paragraph on power network modeling, we remark that a first-principle modeling of a DC power source connected to an AC grid via a droop-controlled  inverter results also in equation \eqref{eq-SI: bus dynamics}; see \citesec{JWSP-FD-FB:12j} for further details.


\subsection{Synchronization Notions} 

The following subsets of the $n$-torus $\mathbb T^{n}$ are essential for the synchronization problem:
For $\gamma\in{[0,\pi/2[}$, let $\bar\Delta_{G}(\gamma) \subset \mbb T^{n}$
be the closed set of angle arrays $(\theta_1,\dots,\theta_n)$ with the
property $|\theta_{i} - \theta_{j}| \leq \gamma$ for $\{i,j\} \in \mc
E$. Also, let $\Delta_{G}(\gamma)$ be the interior of
$\bar\Delta_{G}(\gamma)$.

\begin{definition}\label{Definition: sync}
A solution $\map{(\theta,\dot \theta)}{\real_{\geq0}}{(\mbb T^{n},\mbb
  R^{|\mc V_{1}|})}$ to the coupled oscillator model \eqref{eq-SI: coupled
  oscillator model} is said to be {\em synchronized} if $\theta(0) \in
\bar\Delta_{G}(\gamma)$  and there exists $\subscr{\omega}{sync}\in\real^{1}$ such
that $\theta(t) = \theta(0) + \subscr{\omega}{sync}\fvec 1_n t \pmod{2\pi}$
and $\dot{\theta}(t) =\subscr{\omega}{sync}\fvec 1_{|\mc V_{1}|}$ for all
$t \geq 0$.
\end{definition}
 In other words, here, synchronized trajectories have the
properties of {\em frequency synchronization} and {\em phase cohesiveness},
that is, all oscillators rotate with the same synchronization frequency
$\subscr{\omega}{sync}$ and all their phases belong to the set~$\bar\Delta_{G}(\gamma)$.
For a power network model \eqref{eq-SI: generator dynamics}-\eqref{eq-SI: bus dynamics}, the notion of phase cohesiveness is equivalent to bounded flows $| a_{ij} \sin(\theta_{i}-\theta_{j})| \leq a_{ij} \sin(\gamma)$ for all transmission lines $\{i,j\} \in \mc E$.

For the coupled oscillator model \eqref{eq-SI: coupled oscillator model}, the explicit synchronization frequency is given by
$\subscr{\omega}{sync} \triangleq \sum_{i=1}^{n} \omega_{i}/\sum_{i=1}^{n}
D_{i}$, see \citesec{FD-FB:10w} for a detailed derivation. By transforming to a rotating frame with
frequency $\subscr{\omega}{sync}$ and by replacing $\omega_{i}$ by
$\omega_{i} - D_{i} \subscr{\omega}{sync}$, we obtain
$\subscr{\omega}{sync} = 0$ (or equivalently $\omega \in \fvec 1_n^\perp$)
corresponding to balanced power injections $\sum_{i \in \mc V_{1}} P_{\textup{m},i} = \sum_{i \in \mc V_{2}} P_{\textup{l},i}$
in power network applications. Hence, without loss of generality, we assume that $\omega \in
\fvec 1_n^\perp$ such that $\subscr{\omega}{sync} = 0$.

Given a point $r\in\mycircle$ and an angle $s\in[0,2\pi]$, let
$\rot_s(r)\in\mycircle$ be the rotation of $r$ counterclockwise by the
angle $s$.  For $(r_{1},\dots,r_{n}) \in \torus^n $, define the equivalence
class
\begin{equation*}
  [(r_1,\dots,r_n)] = \setdef{ ( \rot_s(r_1), \dots, \rot_s(r_n) \in \mbb
    T^n }{s\in[0,2\pi]}.
\end{equation*}
Clearly, if $(r_1,\dots,r_n)\in\bar\Delta_G(\gamma)$,  then
$[(r_1,\dots,r_n)]\subset\bar\Delta_G(\gamma)$.
\begin{definition}
Given $\theta\in\bar\Delta_G(\gamma)$ for some $\gamma\in{[0,\pi/2[}$, the
set $([\theta], \fvec 0_{|\mc V_{1}|}) \subset \mbb
T^{n} \times \real^{|\mc V_{1}|}$ is a {\em synchronization manifold} of
the coupled oscillator model~\eqref{eq-SI: coupled oscillator model}.  
\end{definition}
Note that a synchronized solution takes value in a synchronization manifold due to rotational symmetry.
For two first-order oscillators \eqref{eq-SI: Kuramoto model} the state space $\torus^{2}$, the set $\Delta_{G}(\pi/2)$, as well as the synchronization manifold $[\theta^{*}]$ associated to an angle array $\theta^{*} = (\theta_{1}^{*},\theta_{2}^{*}) \in \torus^{2}$ are illustrated in Figure \ref{Fig: sync manifold}.

\begin{figure}[htbp]
	\centering{
	\includegraphics[width = 0.83\columnwidth]{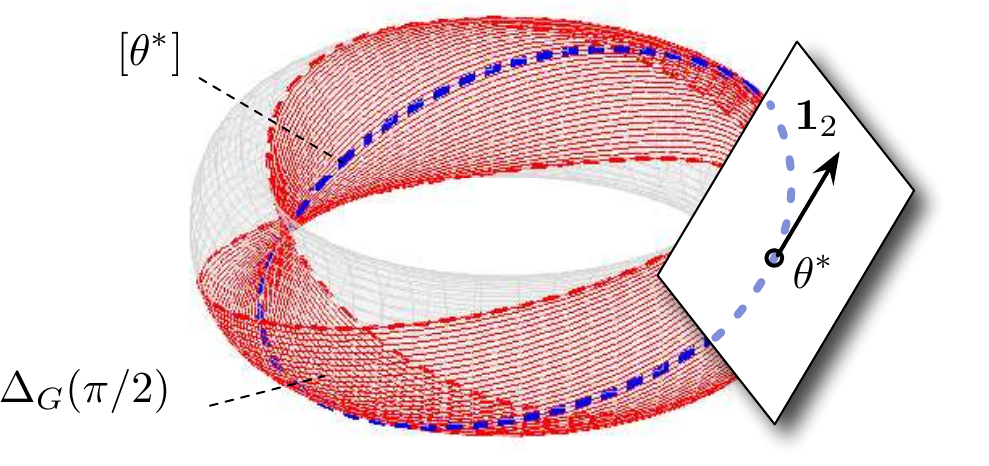}
	\caption{Illustration of the state space $\torus^{2}$, the set $\Delta_{G}(\pi/2)$, the synchronization manifold $[\theta^{*}]$ associated to a point $\theta^{*} = (\theta_{1}^{*},\theta_{2}^{*}) \in \Delta_{G}(\pi/2)$, the tangent space at $\theta^{*}$, and the translation vector\,$\fvec 1_{2}$.}
	\label{Fig: sync manifold}
	}
\end{figure}

\subsection{Existing Synchronization Conditions} 

The coupled oscillator dynamics \eqref{eq-SI: coupled oscillator model}, and
the Kuramoto dynamics \eqref{eq-SI: Kuramoto model} for that matter, feature
(i) the synchronizing effect of the coupling described by the weighted edges of the graph $G(\mc
V,\mc E,A)$ and (ii) the de-synchronizing effect of the dissimilar natural
frequencies $\omega \in \fvec 1_{n}^{\perp}$ at the nodes.  Loosely speaking, synchronization occurs when the
coupling dominates the dissimilarity.  Various conditions are proposed in
the power systems and synchronization literature to quantify this tradeoff between coupling and dissimilarity.
The coupling is typically quantified by the algebraic connectivity $\lambda_{2}(L)$ \citesec{FFW-SK:80,FD-FB:09z,AJ-NM-MB:04,TN-AEM-YCL-FCH:03,AA-ADG-JK-YM-CZ:08,SB-VL-YM-MC-DUH:06} or the weighted nodal degree $\textup{deg}_{i} \triangleq \sum\nolimits_{j=1}^{n} a_{ij}$ \citesec{FW-SK:82,FD-FB:11d,GK-MBH-KEB-MJB-BK-DA:06,FD-FB:09z,LB-LS-ADG:09}, and the dissimilarity is quantified by either absolute norms $\| \omega \|_{p}$ or incremental (relative) norms $\| B^{T} \omega \|_{p}$, where typically $p \in \{2,\infty\}$. Sometimes, these conditions can be evaluated only
numerically since they are state-dependent \citesec{FFW-SK:80,FW-SK:82} or
arise from a non-trivial linearization process, such as the Master stability function formalism
\citesec{AA-ADG-JK-YM-CZ:08,SB-VL-YM-MC-DUH:06,LMP-TLC:98}. In general, concise and
accurate results are only known for specific topologies such as complete
graphs \citesec{FD-FB:10w,MV-OM:08} linear chains \citesec{SHS-REM:88,NK-GBE:88} and complete
bipartite graphs \citesec{MV-OM:09} with uniform weights. 

For arbitrary coupling topologies only sufficient conditions are known \citesec{FFW-SK:80,FD-FB:09z,AJ-NM-MB:04,FW-SK:82} as well as numerical investigations for random networks \citesec{JGG-YM-AA:07,TN-AEM-YCL-FCH:03,YM-AFP:04,ACK:10}. To best of our knowledge, the sharpest and provably correct synchronization conditions for arbitrary topologies assume the form $\lambda_{2}(L) > \left(\sum_{i<j} | \omega_{i} - \omega_{j}|^{2}\right)^{1/2}$, see \citesec[Theorem 4.4]{FD-FB:09z}.
For arbitrary undirected, connected, and weighted, graphs $G(\mc V,\mc E,A)$, simulation studies indicate that the known sufficient conditions \citesec{FFW-SK:80,FD-FB:09z,AJ-NM-MB:04,FW-SK:82} are conservative estimates on the threshold from incoherence to synchrony, and every review article on synchronization concludes with the open problem of finding sharp synchronization conditions \citesec{JAA-LLB-CJPV-FR-RS:05,FD-FB:10w,SHS:00,AA-ADG-JK-YM-CZ:08,SB-VL-YM-MC-DUH:06,SHS:01}.

\section{Mathematical Analysis of Synchronization}
\label{Section: Mathematical Analysis of Synchronization}

This section presents our analysis of the synchronization problem in the coupled oscillator model\,\eqref{eq-SI: coupled oscillator model}.

\subsection{An Algebraic Approach to Synchronization}

Here we present a novel analysis approach that reduces the synchronization problem to an equivalent algebraic problem that reveals the crucial role of cycles and cut-sets in the graph topology. 
In a first analysis step, we reduce the synchronization problem for the coupled oscillator model \eqref{eq-SI: coupled oscillator model} to a simpler problem, namely stability of a first-order model. It turns out that existence and local exponential stability of synchronized solutions of the coupled oscillator model \eqref{eq-SI: coupled oscillator model} can be entirely described by means of the first-order Kuramoto model~\eqref{eq-SI: Kuramoto model}.

\smallskip
\begin{lemma}
{\bf (Synchronization equivalence)}
\label{Lemma: Local Equivalence of First and Second Order Synchronization}
Consider the coupled oscillator model \eqref{eq-SI: coupled oscillator model}
and the Kuramoto model \eqref{eq-SI: Kuramoto model}. The following statements
are equivalent for any $\gamma \in {[0,\pi/2[}$ and any synchronization
    manifold $([\theta], \fvec 0_{|\mc V_{1}|}) \subset \bar\Delta_{G}(\gamma) \times \real^{|\mc V_{1}|}$.
\begin{enumerate}

	\item[(i)] $[\theta]$ is a locally exponentially stable synchronization manifold the Kuramoto model \eqref{eq-SI: Kuramoto model}; and

	\item[(ii)] $([\theta], \fvec 0_{|\mc V_{1}|})$ is a locally exponentially stable synchronization manifold of the coupled oscillator model \eqref{eq-SI: coupled oscillator model}.
	
\end{enumerate}
If the equivalent statements (i) and (ii) are true, then, locally near their respective synchronization manifolds, the coupled oscillator model \eqref{eq-SI: coupled oscillator model} and the Kuramoto model \eqref{eq-SI: Kuramoto model} together with the frequency dynamics $\dt \dot \theta = - M^{-1}D \dot \theta$ are topologically conjugate.
\end{lemma}

Loosely speaking, the topological conjugacy result means that the trajectories of the two plots in Figure \ref{Fig: 1st and 2nd order simulation} can be continuously deformed to match each other while preserving parameterization of time. Lemma \ref{Lemma: Local Equivalence of First and Second Order Synchronization} is illustrated in Figure \ref{Fig: 1st and 2nd order simulation}, and its proof can be found in \citesec[Theorems 5.1 and 5.3]{FD-FB:10w}.

\begin{figure}[h]
	\centering{
	\includegraphics[width = 0.99 \columnwidth]{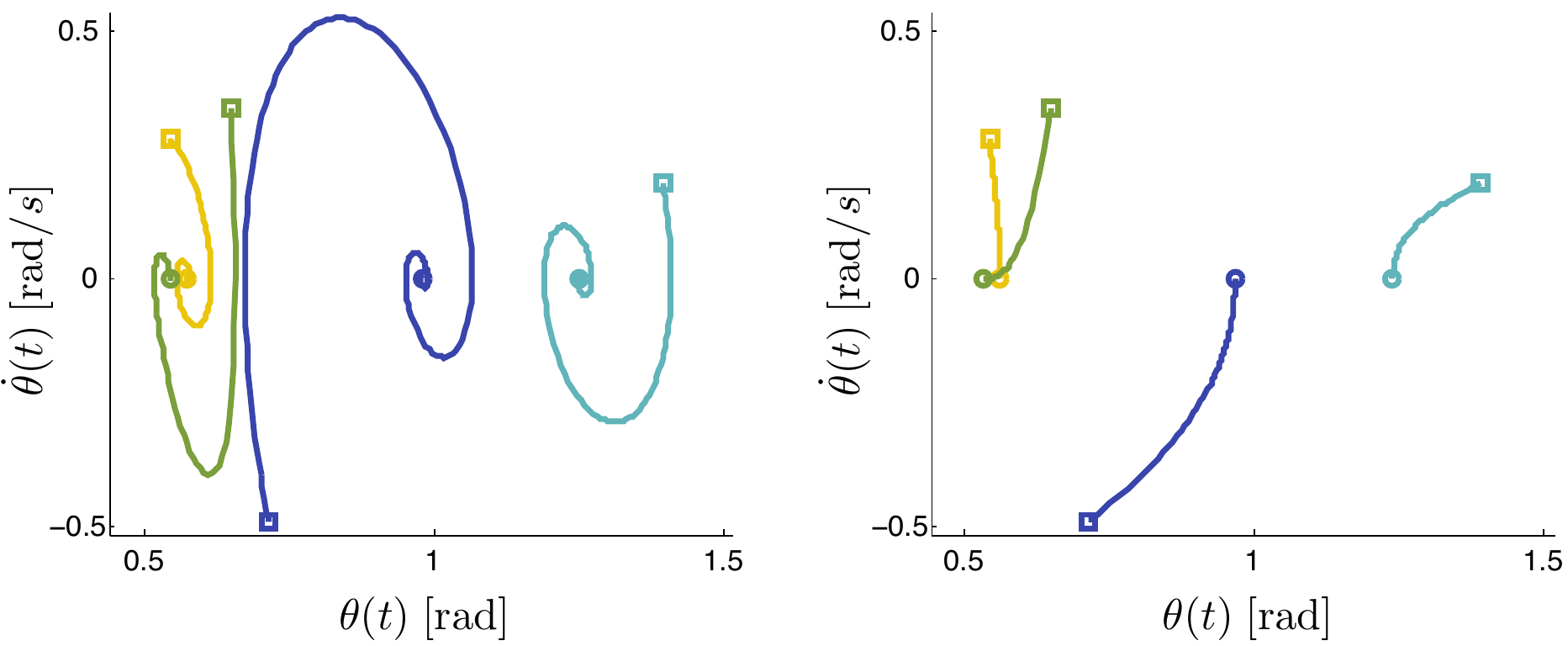}
	\caption{The left plot shows the phase space dynamics of a network of $n=4$ second-order oscillators \eqref{eq-SI: spring network} with $\mc V_{2} = \emptyset$ and Kuramoto-type coupling $a_{ij} = K/n$ for all distinct $i,j \in \mc V_{1} = \until 4$ and for $K \in \real$. The right plot shows the phase space dynamics corresponding to first-order Kuramoto oscillators \eqref{eq-SI: Kuramoto model} together with the frequency dynamics $\dt \dot \theta = - M^{-1}D \dot \theta$. The natural frequencies $\omega_{i}$ and the coupling strength $K$ are chosen such that $\subscr{\omega}{sync} = 0$ and $K = 1.1 \cdot \max_{i,j \in \until 4} | \omega_{i} - \omega_{j}|$. From the same initial configuration $\theta(0)$ (denoted by {\small$\blacksquare$}) both first and second-order oscillators converge exponentially to the same  synchronized equilibria  (denoted by {\Large$\bf\bullet$}), as predicted by Lemma \ref{Lemma: Local Equivalence of First and Second Order Synchronization}. 
	}
	\label{Fig: 1st and 2nd order simulation}
	}
\end{figure}

By Lemma \ref{Lemma: Local Equivalence of First and Second Order Synchronization}, the local synchronization problem for the coupled oscillator model \eqref{eq-SI: coupled oscillator model} reduces to the synchronization problem for the first-order Kuramoto model \eqref{eq-SI: Kuramoto model}. Henceforth, we restrict ourself to the  Kuramoto model \eqref{eq-SI: Kuramoto model}.
The following result is known in the synchronization literature \citesec{AJ-NM-MB:04,FD-FB:09z} as well as in power systems, where the saturation of a transmission line is corresponds to a singularity of the load flow Jacobian resulting in a saddle node bifurcation \citesec{CJT-OJMS:72a,CJT-OJMS:72b,SS-PV:80,ARB-DJH:81,MI:92,AA-SS-VP:81,FW-SK:82,FFW-SK:80,SG-PWS:05,PWS-BCL-MAP:93,ID:92,KSC-DJH:86}. 

\begin{lemma}{\bf (Stable synchronization in $\Delta_{G}(\pi/2)$)}
\label{Lemma: stable fixed point in pi/2 arc}
Consider the Kuramoto model \eqref{eq-SI: Kuramoto model} with a connected graph
$G(\mc V,\mc E,A)$, and let $\gamma \in {[0,\pi/2[}$. The following statements hold:
\begin{enumerate}
\item[1)] {\bf Jacobian:} The Jacobian of the Kuramoto model evaluated at $\theta \in \mathbb T^{n}$ is given by 
\begin{equation*}
	J (\theta) = - B
  \diag(\{a_{ij}\cos(\theta_{i}-\theta_{j})\}_{\{i,j\} \in \mc E}) B^{T}
  \,;
\end{equation*}

\item[2)] {\bf Stability:} If there exists an equilibrium point $\theta^{*} \in
  \bar\Delta_{G}(\gamma)$, then it belongs to a locally exponentially stable
  equilibrium manifold $[\theta^{*}] \in \bar\Delta_{G}(\gamma)$; and
  
\item[3)] {\bf Uniqueness:} This equilibrium manifold is unique in $\bar\Delta_{G}(\gamma)$.
\end{enumerate}
\end{lemma}

\begin{proof}
Since we have that $\frac{\partial}{\partial \theta_{i}} \bigl( \omega_{i} - \sum\nolimits_{k=1}^{n} a_{ik} \sin(\theta_{i}-\theta_{k}) \bigr) = -\sum\nolimits_{k=1}^{n} a_{ik} \cos(\theta_{i} - \theta_{k})$ and $\frac{\partial}{\partial \theta_{j}} \bigl( \omega_{i} - \sum\nolimits_{k=1}^{n} a_{ik} \sin(\theta_{i}-\theta_{k}) \bigr)  = a_{ij}\cos(\theta_{i} - \theta_{j})$, the negative Jacobian of the right-hand side of the Kuramoto model \eqref{eq-SI: Kuramoto model} equals the Laplacian matrix of the connected graph $G(\mc V,\mc E,\tilde A)$ where $\tilde a_{ij} = a_{ij}\cos(\theta_{i}-\theta_{j})$. Equivalently, in compact notation the Jacobian is given by $J (\theta) = - B \diag(\{a_{ij}\cos(\theta_{i}-\theta_{j})\}_{\{i,j\} \in \mc E}) B^{T}$. This completes the proof of statement 1).

The Jacobian $J(\theta)$ evaluated at an equilibrium point $
\theta^{*} \in \bar\Delta_{G}(\gamma)$ is negative semidefinite with rank $n-1$. Its
nullspace is $\fvec 1_{n}$ and arises from the rotational symmetry of the right-hand side of the Kuramoto model \eqref{eq-SI: Kuramoto model}, see Figure \ref{Fig: sync manifold} for an illustration. Consequently,
the equilibrium point $ \theta^{*} \in \bar\Delta_{G}(\gamma)$ is locally
(transversally) exponentially stable. Moreover, the corresponding
equilibrium manifold $[\theta^{*}] \in \bar\Delta_{G}(\gamma)$ is locally exponentially
stable. This completes the proof of statement~2).

The uniqueness statement 3) follows since the right-hand side of \eqref{eq-SI: Kuramoto model} is a one-to-one function for
$\theta \in \bar\Delta_{G}(\pi/2)$, see \citesec[Corollary 1]{AA-SS-VP:81}. 
\end{proof}

By Lemma \ref{Lemma: stable fixed point in pi/2 arc}, the problem of
finding a locally stable synchronization manifold reduces to
that of finding a fixed point $\theta^{*} \in \bar\Delta_{G}(\gamma)$ for
some $\gamma \in {[0,\pi/2[}$. The fixed-point equations of the Kuramoto
    model \eqref{eq-SI: Kuramoto model} read  as
 \begin{equation}
	\omega_{i}
	=
	\sum\nolimits_{j=1}^{n} a_{ij} \sin(\theta_{i}-\theta_{j})
	\,,\quad i \in \until n\,.
	\label{eq-SI: fixed-point equations - component-wise}
 \end{equation}
In a compact notation the fixed-point equations \eqref{eq-SI: fixed-point equations - component-wise} are
\begin{equation}
	\omega 
	=
	B  \diag\left( \{a_{ij}\}_{\{i,j\} \in \mc E} \right) \sinbf(B^{T} \theta)
	\label{eq-SI: fixed-point equations - vector form}
	\,.
\end{equation} 
The following conditions show that the natural frequencies $\omega$ have to be absolutely and incrementally bounded and the nodal degree has to be sufficiently large such that fixed points of \eqref{eq-SI: fixed-point equations - component-wise} exist. 

\begin{lemma}{\bf(Necessary synchronization conditions)}
\label{Lemma: Necessary sync condition}
Consider the Kuramoto model \eqref{eq-SI: Kuramoto model} with graph $G(\mc V,\mc E,A)$ and $\omega \in \fvec 1_n^\perp$. Let  $\gamma \in {[0,\pi/2[}$, and define the weighted nodal degree $\textup{deg}_{i} \triangleq \sum\nolimits_{j=1}^{n} a_{ij}$ for each node $i \in \until n$. The following statements hold:
\begin{enumerate}

	\item[1)] {\bf Absolute boundedness:} If there exists a synchronized solution $\theta \in \bar\Delta_{G}(\gamma)$, then 
	\begin{equation}
	\textup{deg}_{i} \sin(\gamma) 
	\geq |\omega_{i}| \qquad \mbox{ for all } i \in \until n\,.
	\label{eq-SI: necessary sync condition - 1}
\end{equation} 

	\item[2)] {\bf Incremental boundedness:} If there exists a synchronized solution $\theta \in \bar\Delta_{G}(\gamma)$, then 
	\begin{equation}
	(\textup{deg}_{i} + \textup{deg}_{j})\sin(\gamma) 
	\geq |\omega_{i} - \omega_{j}| \qquad \mbox{ for all }  \{i,j\} \in \mc E\,.
	\label{eq-SI: necessary sync condition - 2}
\end{equation} 

\end{enumerate}
\end{lemma}

\begin{proof}
The first condition arises since $\sin(\theta_{i} - \theta_{j}) \in {[-\sin(\gamma),\sin(\gamma)]}$ for $\theta \in \bar\Delta_{G}(\gamma)$, and the fixed-point equation \eqref{eq-SI: fixed-point equations - component-wise} has no solution if condition \eqref{eq-SI: necessary sync condition - 1} is not satisfied. 

Alternatively, since $\omega \in \fvec 1_{n}^{\perp}$, a multiplication of the fixed point equation \eqref{eq-SI: fixed-point equations - vector form} by the vector $(e_{i}^{n} - e_{j}^{n}) \in \fvec 1_{n}^{\perp}$, for $\{i,j\} \in \mc E$, or equivalently a subtraction of the $i$th and $j$th fixed-point equation \eqref{eq-SI: fixed-point equations - component-wise}, yields the following equation for all $ \{i,j\} \in \mc E$:
\begin{equation}
	\omega_{i} - \omega_{j}
	=
	\sum\nolimits_{k=1}^{n} \left( a_{ik} \sin(\theta_{i}-\theta_{k}) - a_{jk} \sin(\theta_{j}-\theta_{k}) \right)
	\,.
	\label{eq-SI: eq-SI: fixed-point equations - component-wise - difference}
\end{equation}
Again, equation \eqref{eq-SI: eq-SI: fixed-point equations - component-wise - difference} has no solution in $\bar\Delta_{G}(\gamma)$ if condition \eqref{eq-SI: necessary sync condition - 2} is not satisfied.
\end{proof}

In the following we aim to find sufficient and sharp conditions under which the fixed-point equations \eqref{eq-SI: fixed-point equations - vector form} admit a solution $\theta^{*} \in \bar\Delta_{G}(\gamma)$. We resort to a rather straightforward solution ansatz. By formally replacing each term $\sin(\theta_{i}-\theta_{j})$ in the fixed-point equations \eqref{eq-SI: fixed-point equations - vector form} by an auxiliary scalar variable $\psi_{ij}$, the fixed-point equation \eqref{eq-SI: fixed-point equations - vector form} is equivalently written~as
\begin{align}
		\omega &= B \diag\left( \{a_{ij}\}_{\{i,j\} \in \mc E} \right)\, \psi 
		\,,
		\label{eq-SI: auxiliary fixed-point equations}\\
		\psi &= \sinbf(B^{T} \theta)
		\label{eq-SI: auxiliary fixed-point equations - back to theta}
		\,,
\end{align}
where $\psi \in \real^{|\mc E|}$ is a vector with elements $\psi_{ij}$. We will refer to equations \eqref{eq-SI: auxiliary fixed-point equations} as the {\em auxiliary-fixed point equation}, and characterize their properties in the following theorem.

\begin{theorem}{\bf(Properties of the fixed point equations)}
\label{Theorem: Properties of the auxiliary-fixed point equations}
Consider the Kuramoto model \eqref{eq-SI: Kuramoto model} with graph $G(\mc V,\mc E,A)$ and $\omega \in \fvec 1_n^\perp$, its fixed-point equations \eqref{eq-SI: fixed-point equations - vector form}, and the auxiliary fixed-point equations \eqref{eq-SI: auxiliary fixed-point equations}. The following statements hold:
\begin{enumerate}
	
	\item[1)] {\bf Exact solution:} Every solution of the auxiliary fixed-point equations \eqref{eq-SI: auxiliary fixed-point equations} is of the form
\begin{equation}
	\psi = B^{T} L^{\dagger} \omega + \subscr{\psi}{hom}
	\label{eq-SI: auxiliary fixed-point equations - solved}
	\,,
\end{equation}
where the homogeneous solution $\subscr{\psi}{hom} \in \real^{|\mc E|}$ satisfies $\diag\left( \{a_{ij}\}_{\{i,j\} \in \mc E} \right)\, \subscr{\psi}{hom} \in \Ker(B)$.
	
	\item[2)] {\bf Exact synchronization condition:} 
	Let  $\gamma \in {[0,\pi/2[}$. The following three statements are equivalent:
	\begin{enumerate}
	
		\item[(i)] There exists a solution $\theta^{*} \in \bar\Delta_{G}(\gamma)$ to the fixed-point equation \eqref{eq-SI: fixed-point equations - vector form};
	
		\item[(ii)] There exists a solution $\theta \in \bar\Delta_{G}(\gamma)$ to 
		\begin{equation}
	B^{T} L^{\dagger} \omega + \subscr{\psi}{hom} = \sinbf(B^{T} \theta)
	\label{eq-SI: non-inverted cycle constraint}
	\,.
\end{equation}
		 for some $\subscr{\psi}{hom} \in \diag\left( \{1/a_{ij}\}_{\{i,j\} \in \mc E} \right) \ker(B)$; and
		
		\item[(iii)] There exists a solution $\psi \in \real^{|\mc E|}$ to the auxiliary fixed-point equation \eqref{eq-SI: auxiliary fixed-point equations} of the form \eqref{eq-SI: auxiliary fixed-point equations - solved} satisfying the {norm constraint} $\| \psi \|_{\infty} \leq \sin(\gamma)$ and the {cycle constraint} $\arcsinbf(\psi) \in \Img(B^{T})$.
	
	\end{enumerate}
	If the three equivalent statements (i), (ii), and (iii) are true, then we
        have the identities $B^{T}\theta^{*} = B^{T}\theta = \arcsinbf(\psi)$. Additionally, $[\theta^{*}] \in \bar\Delta_{G}(\gamma)$ is a locally exponentially stable synchronization manifold.
	
\end{enumerate}
\end{theorem}

\begin{proof}
{\em Statement 1):} Every solution $\psi \in \real^{|\mc E|}$ to the
auxiliary fixed-point equations \eqref{eq-SI: auxiliary fixed-point equations}
is of the form $\psi = \subscr{\psi}{hom} + \subscr{\psi}{pt}$, where
$\subscr{\psi}{hom}$ is the homogeneous solution and $\subscr{\psi}{pt}$ is
a particular solution.
The homogeneous solution satisfies $B \diag\left( \{a_{ij}\}_{\{i,j\} \in
  \mc E} \right) \subscr{\psi}{hom} = \fvec 0_{n}$.
One can easily verify that $\subscr{\psi}{pt} = B^{T} L^{\dagger} \omega$
is a particular solution%
\footnote{
Likewise, it can also be shown that 
$(B \diag( \{a_{ij}\}_{\{i,j\} \in \mc E} ))^{\dagger} \omega$ as well as $ \diag( \{a_{ij}\}_{\{i,j\} \in \mc E} )^{-1} B^{\dagger} \omega$ are other possible particular solutions. All of these solutions differ only by addition of a homogenous solution. Each one can be interpreted as solution to a weighted least squares problem, see \citesec{IAG-IDW:00}. Further solutions can also be constructed in a graph-theoretic way by a spanning-tree decomposition, see \citesec{NB:97}. Our specific choice $\subscr{\psi}{pt} = B^{T} L^{\dagger} \omega$ has the property that $\subscr{\psi}{pt} \in \Img(B^{T})$ lives in the cut-set space, and it is the most useful particular solution in order to proceed with our synchronization analysis.
}%
, since 
$B \diag( \{a_{ij}\}_{\{i,j\} \in \mc E} )
\subscr{\psi}{pt} 
= B \diag( \{a_{ij}\}_{\{i,j\} \in \mc E} ) B^{T}
L^{\dagger} \omega 
= L L^{\dagger} \omega 
= \big(I_{n} - \frac{1}{n} \fvec
1_{n \times n}\big)\omega 
=\omega$.

{\em Statement 2), equivalence} $\bigl( \mbox{(i)} \Leftrightarrow \mbox{(ii)} \bigr):$
If there exists a solution $\theta^{*}$ of the fixed-point equations
\eqref{eq-SI: fixed-point equations - vector form}, then $\theta^{*}$ can be
equivalently obtained from equation \eqref{eq-SI: auxiliary fixed-point
  equations - back to theta} together with the solution \eqref{eq-SI:
  auxiliary fixed-point equations - solved} of the auxiliary equations
\eqref{eq-SI: auxiliary fixed-point equations}. These two equations directly give equation \eqref{eq-SI:
  non-inverted cycle constraint}.

{\em Equivalence} $\bigl( \mbox{(ii)} \Leftrightarrow \mbox{(iii)} \bigr):$
For $\theta^{*} \in \bar\Delta_{G}(\gamma)$, we have from equation \eqref{eq-SI: non-inverted cycle constraint} that $\| \psi \|_{\infty} \leq \sin(\gamma)$ and $\arcsinbf(\psi) = B^{T} \theta^{*}$, that is, $\arcsinbf(\psi) \in \Img(B^{T})$. Conversely, if the norm constraint $\| \psi \|_{\infty} \leq \sin(\gamma)$ and the cycle constraint $\arcsinbf(\psi) \in \Img(B^{T})$ are met, then equation \eqref{eq-SI: non-inverted cycle constraint} is solvable in $\bar\Delta_{G}(\gamma)$, that is, there is $\theta^{*} \in \bar\Delta_{G}(\gamma)$ such that $\arcsinbf(\psi) = B^{T} \theta^{*}$.
The local exponential stability of the associated synchronization manifold $[\theta^{*}]$ follows then directly from Lemma \ref{Lemma: stable fixed point in pi/2 arc}.
\end{proof}

The particular solution $B^{T} L^{\dagger} \omega$ to the auxiliary fixed-point equations \eqref{eq-SI: auxiliary fixed-point equations} lives in the cut-set space $\Ker(B)^{\perp}$ and the homogenous solution $\subscr{\psi}{hom}$ lives in the weighted cycle space $\subscr{\psi}{hom} \in \diag\left( \{1/a_{ij}\}_{\{i,j\} \in \mc E} \right)  \Ker(B)$. As a consequence, by statement (iii) of Theorem \ref{Theorem: Properties of the auxiliary-fixed point equations}, for each cycle in the graph, we obtain one degree of freedom in choosing the homogeneous solution $\subscr{\psi}{hom}$ as well as one nonlinear constraint $c^{T} \arcsinbf(\psi) = 0$, where $c \in \ker(B)$ is a signed path vector corresponding to the cycle.


\begin{remark}{\bf(Comments on necessity)}
\normalfont
The cycle space $\Ker(B)$ of the graph serves as a degree of freedom to find a minimum $\infty$-norm solution $\psi^{*}$ to equations \eqref{eq-SI: auxiliary fixed-point equations} via
\begin{equation}
  \min\nolimits_{\psi \in \real^{|\mc E|}} \|\psi \|_{\infty} \quad
  \mbox{ subject to} \quad \omega = B \diag\left( \{a_{ij}\}_{\{i,j\} \in \mc E} \right)\, \psi.
  	\label{eq-SI: optimal necessary condition}.
\end{equation}
By Theorem \ref{Theorem: Properties of the auxiliary-fixed point
  equations}, such a minimum $\infty$-norm solution $\psi^{*}$ necessarily
satisfies $\| \psi^{*} \|_{\infty} \leq \sin(\gamma)$ so that an
equilibrium $\theta^{*} \in \bar\Delta_{G}(\gamma)$ exists. Hence, the
condition $\| \psi^{*} \|_{\infty} \leq \sin(\gamma)$ is an {\em optimal
  necessary synchronization condition}.
  
  The optimization problem \eqref{eq-SI: optimal necessary condition} --  the minimum $\infty$-norm solution to an under-determined and consistent system of linear equations -- is well studied in the context of kinematically redundant manipulators. Its solution is known to be non-unique and contained in a disconnected  
solution space \citesec{IAG-IDW:00,IH-JL:02}. Unfortunately, there is no ``a priori'' analytic formula to construct a minimum $\infty$-norm solution, but the optimization problem is computationally tractable via its dual problem 
$\max_{u \in \real^{n}} u^{T} \omega$ subject to $\| \diag\left( \{a_{ij}\}_{\{i,j\} \in \mc E} \right)\, B^{T} u \|_{1} = 1$.
%
%
%
\oprocend
\end{remark}


\subsection{Synchronization Assessment for Specific Networks}

In this subsection we seek to establish that the condition
\begin{equation}
	\boxed{\Bigl.\;
	 \left\| B^{T} L^{\dagger} \omega \right\|_{\infty} = \left\| L^{\dagger} \omega \right\|_{\mc E,\infty}  < 1
	 \;}
	\label{eq-SI: sync condition}
\end{equation}
is sufficient for the existence of locally exponentially stable equilibria
in $\Delta_{G}(\pi/2)$. More general, for a given level of phase cohesiveness $\gamma \in {[0,\pi/2[}$ we seek to establish that the condition
\begin{equation}
	\boxed{	\Bigl.\;
	 \left\| B^{T} L^{\dagger} \omega \right\|_{\infty} = \left\| L^{\dagger} \omega \right\|_{\mc E,\infty} \leq \sin(\gamma)
	 \;}
	\label{eq-SI: sync condition - gamma}
\end{equation}
is sufficient for the existence of locally exponentially stable equilibria
in $\bar\Delta_{G}(\gamma)$. Since the right-hand side of \eqref{eq-SI: sync
  condition - gamma} is a concave function of $\gamma \in {[0,\pi/2[}$ that
    achieves its supremum value at $\gamma^{*} = \pi/2$, it follows that
    condition \eqref{eq-SI: sync condition - gamma} implies\,\eqref{eq-SI: sync
      condition}. 
      
In the main article, we provide a detailed interpretation of the synchronization conditions \eqref{eq-SI: sync condition} and \eqref{eq-SI: sync condition - gamma} from various practical perspectives. Before continuing our theoretical analysis, we provide two further abstract but insightful perspectives on the conditions \eqref{eq-SI: sync condition} and \eqref{eq-SI: sync condition - gamma}.

\begin{remark}{\bf(Interpretation of the sync condition)}
\label{Remark: interpretation of sync condition}
\normalfont\\
{\em Graph-theoretic interpretation:}
With regards to the exact and state-dependent norm and cycle conditions in statement (iii) of Theorem \ref{Theorem: Properties of the auxiliary-fixed point equations}, the proposed condition \eqref{eq-SI: sync condition - gamma} is simply a norm constraint on the network parameters in cut-set space $\Img(B^{T})$ of the graph topology, and cycle components are discarded.

{\em Circuit-theoretic interpretation:} In a circuit or power network, the variable $\omega \in \real^{n}$ corresponds to nodal power injections. Let $x \in \mathbb R^{|\mc E|}$ satisfy $Bx = \omega$, then $x$ corresponds to equivalent power injections along lines $\{i,j\} \in \mc E$.%
\footnote{
Notice that $x$ is not uniquely determined if the circuit features loops.}
Condition \eqref{eq-SI: sync condition} can then be rewritten as $\left\| B^{T} L^{\dagger} B x \right\|_{\infty} < 1$. The matrix $B^{T} L^{\dagger} B \in \mathbb R^{|\mc E| \times |\mc E|}$ has elements $(e_{n}^{i} - e_{n}^{j})^{T} L^{\dagger} (e_{n}^{k} - e_{n}^{\ell})$ for $\{i,j\}, \{k,\ell\} \in \mc E$, its diagonal elements are the effective resistances $R_{ij}$, and its off-diagonal elements are the network distribution (sensitivity) factors \citesec[Appendix 11A]{AJW-BFW:96}. Hence, from a circuit-theoretic perspective condition \eqref{eq-SI: sync condition} restricts the pair-wise effective resistances and the routing of power through the network similar to the resistive synchronization conditions developed in \citesec{FW-SK:82,FD-FB:11d,GK-MBH-KEB-MJB-BK-DA:06}
\oprocend
\end{remark}

As it turns out, the exact state-dependent synchronization conditions in Theorem \ref{Theorem: Properties of the auxiliary-fixed point equations} can be easily evaluated for the sparsest (acyclic) and densest (homogeneous) topologies and for ``worst-case'' (cut-set inducing) and ``best'' (identical) natural frequencies. For all of these cases the scalar condition \eqref{eq-SI: sync condition - gamma} is sharp. To quantify a ``sharp'' condition in the following theorem, we distinguish between {\em exact} (necessary and sufficient) conditions and {\em tight} conditions, which are sufficient in general and become necessary over a set of parametric realizations.

\begin{theorem}{\bf(Sync condition for extremal network topologies and parameters)}
\label{Theorem: Exact and tight synchronization conditions for extremal graphs}
Consider the Kuramoto model \eqref{eq-SI: Kuramoto model} with connected graph $G(\mc V,\mc E,A)$ and $\omega \in \fvec 1_n^\perp$. Consider the inequality condition \eqref{eq-SI: sync condition - gamma} for $\gamma \in {[0,\pi/2[}$.\\
The following statements hold:
\begin{enumerate}

	\item[(G1)] {\bf Exact synchronization condition for acyclic graphs:} Assume that $G(\mc V,\mc E,A)$ is acyclic. There exists an exponentially stable equilibrium $\theta^{*} \in \bar\Delta_{G}(\gamma)$ if and only if condition \eqref{eq-SI: sync condition - gamma} holds. Moreover, in this case we have that $B^{T} \theta^{*} = \arcsinbf(B^{T}L^{\dagger} \omega) \in \bar\Delta_{G}(\gamma)$;

	\item[(G2)] {\bf Tight synchronization condition for homogeneous graphs:} Assume that $G(\mc V,\mc E,A)$ is a homogeneous graph, that is, there is $K>0$ such that $a_{ij} = K$ for all distinct $i,j \in \until{n}$. Consider a compact interval $\Omega \subset \mathbb R$, and let $\fvec\Omega \subset \mathbb R^{n}$ be the set of all vectors with components $\fvec\Omega_{i} \in \Omega$ for all $i \in \until n$. For all $\omega \in \fvec\Omega$ there exists an exponentially stable equilibrium $\theta^{*} \in \bar\Delta_{G}(\gamma)$ if and only if condition \eqref{eq-SI: sync condition - gamma} holds;
	
	\item[(G3)] {\bf Exact synchronization condition for cut-set
          inducing natural frequencies:} Let $\Omega_{1},\, \Omega_{2} \in
          \mathbb R$, and let $\fvec\Omega \subset \mathbb R^{n}$ be the set of bipolar
          vectors with components $\fvec\Omega_{i} \in \{
          \Omega_{1},\Omega_{2}\}$ for $i \in \until n$. For all $\omega
          \in L \fvec\Omega$ there exists an
          exponentially stable equilibrium $\theta^{*} \in
          \bar\Delta_{G}(\gamma)$ if and only if condition \eqref{eq-SI: sync
            condition - gamma} holds.
	Moreover, $\fvec\Omega$ induces a cut-set: if $|\Omega_{2} - \Omega_{1}| = \sin(\gamma)$, then for $\omega = L\fvec\Omega$ we obtain the stable 
	equilibrium $\theta^{*} \in \bar\Delta_{G}(\gamma)$ satisfying $B^{T} \theta^{*} = \arcsin(B^{T}\fvec \Omega)$, that is,
	for all $\{i,j\} \in \mc E$, $|\theta_{i}^{*} - \theta_{j}^{*}|= 0$ if $\fvec\Omega_{i} = \fvec\Omega_{j}$ and  $|\theta_{i}^{*} - \theta_{j}^{*}|= \gamma$ if $\fvec\Omega_{i} \neq \fvec\Omega_{j}$; and
	
	\item[(G4)] {\bf Asymptotic correctness:} In the limit $\omega \to \fvec 0_{n}$, there exists an exponentially stable equilibrium $\theta^{*} \in \bar\Delta_{G}(\gamma)$ if and only if condition \eqref{eq-SI: sync condition - gamma} holds. Moreover, for each component $i \in \until{|\mc E|}$, we have that $\lim_{\omega \to \fvec 0_{n}} \bigl( B^{T} \theta^{*} \bigr)_{i} / \bigl( \arcsinbf(B^{T}L^{\dagger} \omega) \bigr)_{i} = 1$.

\end{enumerate}
\end{theorem}

\begin{proof}{\em Statement (G1):} For an acyclic graph we have that $\Ker(B) = \emptyset$. According to Theorem \ref{Theorem: Properties of the auxiliary-fixed point equations}, there exists an equilibrium $\theta^{*} \in \bar \Delta_{G}(\gamma)$ if and only if condition \eqref{eq-SI: sync condition - gamma} is satisfied. In this case, we obtain $B^{T}\theta^{*} = \arcsinbf(B^{T} L^{\dagger} \omega)$. This completes the proof of statement (G1).

{\em Statement (G2):} In the homogeneous case, we have that $L = K \bigl(n I_{n} - \fvec 1_{n \times n} \bigr)$ and $L^{\dagger} = \frac{1}{Kn} \bigl(I_{n} - \frac{1}{n} \fvec 1_{n \times n} \bigr)$, see \citesec[Lemma 3.13]{FD-FB:11d}. Thus, the inequality condition \eqref{eq-SI: sync condition - gamma} can be equivalently rewritten as
$
	\sin(\gamma) \geq \left\| B^{T} L^{\dagger} \cdot \omega \right\|_{\infty} = \frac{1}{Kn} \left\| B^{T} \omega \right\|_{\infty} 
$.
According to \citesec[Theorem 4.1]{FD-FB:10w}, the Kuramoto model \eqref{eq-SI: Kuramoto model} with homogenous coupling $a_{ij} = K$ features an exponentially stable equilibrium $\theta^{*} \in \bar\Delta_{G}(\gamma)$, $\gamma \in {[0,\pi/2[}$, for all $\omega \in \fvec \Omega$ if and only if the condition $K >  \left\| B^{T} \omega \right\|_{\infty}/n$ is satisfied. This concludes the proof of statement (G2).

{\em Statement (G3):} For notational convenience, let $c \triangleq \subscr{\Omega}{1} - \subscr{\Omega}{2}$. Then, for $\omega \in L\fvec\Omega$, we have that $B^{T} L^{\dagger} \omega = B^{T} L^{\dagger} L \fvec\Omega = B^{T} \fvec\Omega$ is a vector with components $\{-c,0,+c\}$. Now consider the solution $\psi = B^{T}
L^{\dagger} \omega = B^{T} \fvec\Omega$ to the auxiliary fixed point equations \eqref{eq-SI: auxiliary fixed-point equations}, and notice that $\arcsinbf(\psi) = \arcsinbf(B^{T} \fvec \Omega)$ has components $\{-\arcsin(c),0,+\arcsin(c)\}$. In
particular, we have that $\arcsinbf(\psi) \in \Img(B^{T})$, and the exact
synchronization condition from Theorem \ref{Theorem: Properties of the
  auxiliary-fixed point equations} is satisfied if and only if
$\|\psi\|_{\infty} = c \leq \sin(\gamma)$, which corresponds to condition
\eqref{eq-SI: sync condition - gamma}. The cut-set property follows since $B^{T}
\theta^{*} = \arcsinbf(\psi)$ has components $\{-\arcsin(c),0,+\arcsin(c)\} =
\{- \gamma,0,+\gamma\}$.
This concludes the proof of statement (G3).

{\em Statement (G4):} Since $\lim_{x \to 0} \bigl(\arcsin(x)/x\bigr) = 1$ for $x \in \real$, we obtain 
%
$\lim_{\omega \to \fvec 0_{n}}
\bigl( \arcsinbf(B^{T} L^{\dagger} \omega)_{i} / (B^{T} L^{\dagger} \omega) \bigr)_{i} = 1$ for each component $i \in \until{|\mc E|}$. Thus, the cycle constraint $\arcsinbf(\psi) \in \Img(B^{T})$
is asymptotically met with $\psi = B^{T} L^{\dagger} \omega$. In this case,
the solution of equation \eqref{eq-SI: non-inverted cycle constraint} is
obtained as $B^{T} \theta^{*} = B^{T} L^{\dagger}\omega$, and we have that
$\theta^{*} \in \bar\Delta_{G}(\gamma)$ if and only if the norm constraint
\eqref{eq-SI: sync condition - gamma} is satisfied.%
\footnote{Of course, the limit $\omega \to \fvec 0_{n}$ also implies that the resulting equilibrium $\theta^{*} \in \bar\Delta_{G}(0)$ corresponds to phase synchronization $\theta_{i} = \theta_{j}$ for all $i,j \in \until n$. The converse statement $\theta^{*} \in \bar\Delta_{G}(0)$ $\implies$ $\omega = \fvec 0_{n}$ is also true and its proof can be found in \citesec[Theorem 5.5]{FD-FB:10w}.
}
This concludes the proof of statement (G4) and Theorem \ref{Theorem: Exact and tight synchronization conditions for extremal graphs}.
\end{proof}

Theorem \ref{Theorem: Properties of the auxiliary-fixed point equations} shows that the solvability of the fixed-point equations \eqref{eq-SI: fixed-point equations - vector form} is inherently related to the {\em cycle constraints}. The following lemma establishes feasibility of a single cycle.

\begin{lemma}[\bf Single cycle feasibility]
\label{Lemma: cycle feasibility}
Consider the Kuramoto model \eqref{eq-SI: Kuramoto model} with a cycle graph
$G(\mc V,\mc E,A)$ and $\omega \in \fvec 1_n^\perp$.
Without loss of generality, assume that the edges are labeled by $\{i,i+1\} \pmod n$ for $i \in \until n$ and $ \Ker(B) = \textup{span}(\fvec 1_{n})$. Define
 $x \in \fvec 1_n^\perp$ and $y \in \mathbb R^{n}_{>0}$ uniquely by $x \triangleq B^{T} L^{\dagger} \omega$ and $y_{i} \triangleq a_{i,(i+1)\!\pmod n} > 0$ for $i \in \until n$.
Let $\gamma \in {[0,\pi/2[}$. \\The following statements are equivalent:
\begin{enumerate}

\item[(i)] There exists a stable equilibrium $\theta^{*} \in \bar\Delta_{G}(\gamma)$; and 

\item[(ii)] The function
  $\map{f}{[\subscr{\lambda}{min},\subscr{\lambda}{max}]}{\real}$ with domain boundaries
    $\subscr{\lambda}{min} = \max\limits_{i \in \until
    n} \frac{-\sin(\gamma) - x_{i}}{y_{i}}$
    and
      $\subscr{\lambda}{max} = \min\limits_{i \in \until n} \frac{\sin(\gamma)
    - x_{i}}{y_{i}}$ 
    and defined by 
  $f(\lambda) = \sum_{i=1}^{n} \arcsin(x_{i} + \lambda y_{i})$ satisfies
  $f(\subscr{\lambda}{min}) < 0 < f(\subscr{\lambda}{max})$. 

\end{enumerate}
If both equivalent statements 1) and 2) are true, then $B^{T} \theta^{*} \!=\! \arcsinbf(x + \lambda^{*} y) $, where $\lambda^{*} \in [\subscr{\lambda}{min},\subscr{\lambda}{max}]$ satisfies $f(\lambda^{*}) \!=\!0$.
\end{lemma}

\begin{proof}
According to Theorem \ref{Theorem: Properties of the auxiliary-fixed point equations}, there exists a stable equilibrium $\theta^{*} \in \bar\Delta_{G}(\gamma)$ if and only if there exists a solution $\psi = x + \lambda y$, $\lambda \in \mathbb R$, to the auxiliary fixed-point equations \eqref{eq-SI: auxiliary fixed-point equations}
satisfying the norm constraint $\| \psi \|_{\infty} \leq \sin(\gamma)$ and the cycle constraint $\arcsinbf(\psi) \in \Img(B^{T})$. 

Equivalently, since $ \Ker(B) = \textup{span}(\fvec 1_{n})$, there is $\lambda \in \mathbb R$ satisfying the norm constraint
$\|x + \lambda y\|_{\infty} \leq \sin(\gamma) < 1$ and the cycle constraint $\fvec 1_{n}^{T}\arcsinbf(x + y \lambda) = 0$.
%
Equivalently, the function $f(\lambda) = \fvec 1_{n}^{T}\arcsinbf(x + y \lambda)$ features a zero $\lambda^{*} \in [\subscr{\lambda}{min},\subscr{\lambda}{max}]$ (corresponding to the cycle constraint), where the constraints on $\subscr{\lambda}{min}$ and $\subscr{\lambda}{max}$ guarantee the norm constraints $x_{i} + y_{i} \subscr{\lambda}{max} \leq \sin(\gamma)$ and $x_{i} + y_{i} \subscr{\lambda}{min} \geq -\sin(\gamma)$ for all $i \in \until n$.
Equivalently, by the intermediate value theorem and due to continuity and (strict) monotonicity of the function $f$, we have that $f(\subscr{\lambda}{min}) < 0 < f(\subscr{\lambda}{max})$.
Finally, if $\lambda^{*} \in [\subscr{\lambda}{min},\subscr{\lambda}{max}]$ is found such that $f(\lambda^{*}) = 0$, then,\,by Theorem \ref{Theorem: Properties of the auxiliary-fixed point equations}, $B^{T} \theta^{*} = \arcsinbf(\psi) = \arcsinbf(x + \lambda^{*}y)$. 
\end{proof}

Lemma \ref{Lemma: cycle feasibility} offers a checkable synchronization condition for cycles, which leads to the following theorem.

\begin{theorem}{\bf( Sync conditions for cycle graphs)}
\label{Theorem: cycle graph}
Consider the Kuramoto model \eqref{eq-SI: Kuramoto model} with a cycle graph $G(\mc V,\mc E,A)$  and $\omega \in \fvec 1_n^\perp$. Consider the inequality condition \eqref{eq-SI: sync condition - gamma} for $\gamma \in {[0,\pi/2[}$. The following statements hold.
\begin{enumerate}
	
	\item[(C1)] {\bf Exact sync condition for symmetric natural frequencies:} Assume that $\omega \in \fvec 1_n^\perp$ is such that $B^{T} L^{\dagger}  \omega$ is a symmetric vector 
	\footnote{A vector $x \in \fvec 1_n^\perp$ is {\em symmetric} if its histogram is symmetric, that is, up to permutation of its elements, $x$ is of the form $x = [-c,+c]^{T}$ for $n$ even and some vector $c \in \mathbb R^{n/2}$ and $x = [-c,0,+c]^{T}$ for $n$ odd and some $c \in \mathbb R^{(n-1)/2}$.}%
	. There is an exponentially stable equilibrium $\theta^{*} \in \bar\Delta_{G}(\gamma)$ if and only if condition \eqref{eq-SI: sync condition - gamma}\,holds. Moreover,  in this case $B^{T} \theta^{*} \!=\! \arcsinbf(B^{T}L^{\dagger} \omega)$. 
	
		\item[(C2)] {\bf Tight sync condition for low-dimensional cycles:} Assume the network contains $n \in \{3,4\}$ oscillators. Consider a compact interval $\Omega \subset \mathbb R$, and let $\fvec\Omega \in \mathbb R^{n}$ be the set of vectors with components $\fvec\Omega_{i} \in \Omega$ for all $i \in \until n$. For all $\omega \in L\fvec\Omega$ there exists an exponentially stable equilibrium $\theta^{*} \in \bar\Delta_{G}(\gamma)$ if and only if condition \eqref{eq-SI: sync condition - gamma} holds. 
		
		\item[(C3)] {\bf General cycles and network parameters:} In general for $n \geq 5$ oscillators, condition \eqref{eq-SI: sync condition} does \underline{not} guarantee existence of an equilibrium $\theta^{*} \in \Delta_{G}(\pi/2)$. As a sufficient condition, there exists an exponentially stable equilibrium $\theta^{*} \in \bar\Delta_{G}(\gamma)$, $\gamma \in {[0,\pi/2[}$\,, if 

\end{enumerate}
\begin{equation}
	\left\| B^{T} L^{\dagger} \omega \right\|_{\infty}  \leq 
	\frac{\min_{\{i,j\} \in \mc E} a_{ij}}{ \max_{\{i,j\} \in \mc E} a_{ij} + \min_{\{i,j\} \in \mc E} a_{ij} } \cdot \sin(\gamma)
	\label{eq-SI: sufficient condition for cyclic equilibrium}
	\,.
\end{equation}
\end{theorem}
\medskip

\begin{proof}
To prove the statements of Theorem \ref{Theorem: cycle graph} and to show the existence of an equilibrium $\theta^{*} \in \bar\Delta_{G}(\gamma)$, we invoke the equivalent formulation via the function $f(\lambda)$ as constructed in Lemma \ref{Lemma: cycle feasibility}. In particular, we seek to prove the statement:

\begin{quotation}
 Let 
    $\subscr{\lambda}{min} = \max\nolimits_{i \in \until
    n} \frac{-\sin(\gamma) - x_{i}}{y_{i}}$
    and
     $\subscr{\lambda}{max} = \min\nolimits_{i \in \until n} \frac{\sin(\gamma)
    - x_{i}}{y_{i}}$. 
    The function
  $\map{f}{[\subscr{\lambda}{min},\subscr{\lambda}{max}]}{\real}$ defined
  by $f(\lambda) = \sum_{i=1}^{n} \arcsin(x_{i} + \lambda y_{i})$ satisfies
  $f(\subscr{\lambda}{min}) < 0 < f(\subscr{\lambda}{max})$ (equivalently there is $\lambda^{*} \in [\subscr{\lambda}{min},\subscr{\lambda}{max}]$ such that $f(\lambda^{*}) = 0$) if and only if the condition $\|x\|_{\infty} \leq \sin(\gamma)$ is satisfied.    
\end{quotation}

{\em Statement (C1):}
For a symmetric vector $x = B^{T} L^{\dagger} \omega$, all odd moments about the (zero) mean vanish, that is, $\sum_{i=1}^{n} x_{i}^{2p+1} = 0$ for $p \in \mathbb N_{0}$. Since the Taylor series of the $\arcsin$ about zero features only odd powers, we have $f(0) = \sum_{i=1}^{n} \arcsin(x_{i}) = \sum_{i=1}^{n} \sum_{p=0}^\infty  \frac {(2p)!} {2^{2p}(p!)^2 (2p+1)} x_{i}^{2p+1} = 0$. Statement 1) follows then immediately from Lemma \ref{Lemma: cycle feasibility}.

{\em Statement (C2):} 
By statement (C1), statement (C2) is true if $B^{T} L^{\dagger} \omega$ is symmetric. Statement (C2), can then be proved in a combinatorial fashion by considering all deviations from symmetry arising for three or four oscillators.
In order to continue recall that $\arcsin(x)$ is a super-additive function for $x \in {[0,1]}$ and a sub-additive function for $x \in {[-1,0]}$, that is, $\arcsin(x) + \arcsin(y) < \arcsin(x+y)$ for $x,y > 0$ and $x+y \leq 1$, $\arcsin(x) + \arcsin(y) > \arcsin(x+y)$ for $x,y < 0$ and $x + y \geq -1$, and $\arcsin(x) + \arcsin(y) = \arcsin(x+y)$ for $x=y=0$. We now consider each case $n \in \{3,4\}$ separately.

{\it Proof of sufficiency for} $n=3$: Assume that $\|x\|_{\infty} \leq \sin(\gamma)$. Since the case $f(\lambda = 0) = \fvec 1_{n}^{T}\arcsinbf(x) = 0$ for a symmetric vector $x \in \mathbb R^{3}$ is already proved, we consider now the asymmetric  case $f(\lambda = 0) = \fvec 1_{n}^{T}\arcsinbf(x) > 0$ (the proof of the case $\fvec 1_{n}^{T}\arcsinbf(x) < 0$ is analogous). Necessarily, it follows that at least two elements of $x$ are negative: if one element of $x$ is zero, say $x_{1} = 0$, then we fall back into the symmetric case $x_{2} = -x_{3}$; on the other hand, if only one element is negative, say $x_{1} < 0$ and $x_{2},x_{3} > 0$, then we arrive at a contradiction since $f(\lambda = 0) = \sum_{i=1}^{n} \arcsin(x_{i}) = - \arcsin(x_{2} + x_{3}) + \arcsin(x_{2}) + \arcsin(x_{3}) < 0$ due to super-additivity and since $x_{1} = -x_{2}-x_{3}$.
Hence, without loss of generality, let $x = [a+b,-a,-b]^{T}$ where $a,b > 0$. 
By assumption $\|x\|_{\infty} \leq \sin(\gamma)$. It follows that $a+b \leq \sin(\gamma)$, $a < \sin(\gamma)$, $b < \sin(\gamma)$, and $\subscr{\lambda}{min} = \max\nolimits_{i \in \until
    n} \frac{-\sin(\gamma) - x_{i}}{y_{i}} < 0$. 

Due to super-additivity,  $f(\lambda = 0) = \fvec 1_{n}^{T}\arcsinbf(x)  = \arcsin(a+b) - (\arcsin(a) + \arcsin(b)) > 0$. 
Now we evaluate $f(\lambda)$ at the lower end of its domain $[\subscr{\lambda}{min},\subscr{\lambda}{max}]$ and obtain
\begin{align}
	f(\subscr{\lambda}{min}) =&\; 
	\arcsin(a+b + y_{1} \subscr{\lambda}{min}) + \arcsin(-a + y_{2}\subscr{\lambda}{min}) 
	\nonumber\\&\;
	+ \arcsin(-b + y_{3}\subscr{\lambda}{min})
	\label{eq-SI: f for n=3}
	\,.
\end{align}
By the definition of $\subscr{\lambda}{min}$, at least one summand on the right-hand side of \eqref{eq-SI: f for n=3} equals $-\gamma$. Furthermore, notice that the second and the third summand are negative, and the first summand satisfies $\arcsin(a+b + y_{1}\subscr{\lambda}{min}) \geq -\gamma$. If $\arcsin(a+b + y_{1}\subscr{\lambda}{min}) = -\gamma$, then clearly $f(\subscr{\lambda}{min}) < 0$. In the other case, $\arcsin(a+b + y_{1}\subscr{\lambda}{min}) > -\gamma$, it follows that
\begin{multline*}
	f(\subscr{\lambda}{min}) < 
	\underbrace{\arcsin(a+b+ y_{1}\subscr{\lambda}{min}) - \gamma}_{< 0} 
	\\ + \underbrace{\max \bigl\{\arcsin(-a+ y_{2}\subscr{\lambda}{min}), \arcsin(-b+ y_{3}\subscr{\lambda}{min})\bigr\}}_{< 0} < 0
	\,.
\end{multline*}
Since $f(\subscr{\lambda}{min}) < 0 < f(0) \leq f(\subscr{\lambda}{max})$, it follows from Lemma \ref{Lemma: cycle feasibility} that there exists a stable equilibrium $\theta^{*} \in \bar\Delta_{G}(\gamma)$. The sufficiency is proved for $n = 3$.

{\it Proof of sufficiency} for $n=4$: Assume that $\|x\|_{\infty} \leq \sin(\gamma)$. Without loss of generality, let $\argmax_{i \until 4}\{|x_{i}|\}$ be a singleton (otherwise $x$ is necessarily symmetric), and let $x \in \fvec 1_n^\perp$ be such that $f(\lambda = 0) = \fvec 1_{n}^{T}\arcsinbf(x) > 0$ (the proof of the case $\fvec 1_{n}^{T}\arcsinbf(x) < 0$ is analogous). 
Necessarily, it follows that at least two elements of $x$ are negative: if only one element of $x$ is negative, say $x_{1} < 0$ and $x_{2},x_{3},x_{4} \geq 0$, then we arrive at a contradiction since $f(\lambda = 0) = \sum_{i=1}^{n} \arcsin(x_{i}) = - \arcsin(x_{2} + x_{3} + x_{4}) + \arcsin(x_{2}) + \arcsin(x_{3}) + \arcsin(x_{4})$ is zero only in the symmetric case (for example, $x_{2}=x_{3} = 0 < x_{4} = -x_{1}$) and strictly negative otherwise (due to super-additivity). 
If exactly one element of $x$ is positive (and three are non-postive), say $x = [a+b+c,-a,-b,-c]^{T}$ for $a,b,c \geq 0$ and $a+b + c = \|x\|_{\infty}  \leq \sin(\gamma)$, then, an analogous reasoning to the case $n=3$ leads to $f(\subscr{\lambda}{min})  < 0$.

It remains to consider the case of two positive and two negative entries. Without loss of generality let $x_{1} \geq x_{2} > 0 > x_{3} \geq x_{4}$, where $x_{1} \neq -x_{4}$ and $x_{2} \neq -x_{3}$ (this is the symmetric case), $\sum_{i=1}^{n} x_{i} = 0$, and $\|x\|_{\infty}  \leq \sin(\gamma)$ by assumption. It follows that $\subscr{\lambda}{min} = \max\nolimits_{i \in \until n} \frac{-\sin(\gamma) - x_{i}}{y_{i}} \leq 0$.
Since $f(\lambda = 0) = \fvec 1_{n}^{T} \arcsin(x) > 0$ and $\fvec 1_{n}^{T} x = 0$, it follows from super-additivity that $\|x\|_{\infty} = \max\{x_{1},x_{2}\}$, and the set $\argmax\{x_{1},x_{2}\}$ must be a singleton (otherwise we arrive again at a contradiction or at the symmetric case). Suppose that $\|x\|_{\infty} = \max\{x_{1},x_{2}\} = x_{1}$, then necessarily $|x_{2}| < |x_{3}| \leq |x_{4}| < |x_{1}| \leq \sin(\gamma)$. 
It follows that $\subscr{\lambda}{min}<0$.

%
Again, we evaluate the sum
$f(\subscr{\lambda}{min}) = \sum_{i=1}^{4}\arcsin(x_{i} + y_{i} \subscr{\lambda}{min})$. Notice that the last two summands $\arcsin(x_{3} + y_{3}\subscr{\lambda}{min})$ and $\arcsin(x_{4} + y_{4}\subscr{\lambda}{min})$ are negative (since $0 > x_{3} \geq x_{4}$ and $\subscr{\lambda}{min} < 0$), and the first two summands satisfy $\min \bigl\{\arcsin(x_{1} + y_{1}\subscr{\lambda}{min}),\arcsin(x_{2} + y_{2} \subscr{\lambda}{min}) \bigr\} \geq - \gamma$.
If
$\min \bigl\{\arcsin(x_{1} + y_{1}\subscr{\lambda}{min}),\arcsin(x_{2} + y_{2} \subscr{\lambda}{min}) \bigr\} = - \gamma$, 
we have 
\begin{multline*}
	f(\subscr{\lambda}{min}) =
 \underbrace{\arcsin(x_{3} + y_{3}\subscr{\lambda}{min}) + \arcsin(x_{4} + y_{4}\subscr{\lambda}{min})}_{< 0}\\
+  \underbrace{ \bigl(- \gamma + \max \bigl\{\arcsin(x_{1} + y_{1}\subscr{\lambda}{min}),\arcsin(x_{2} + y_{2} \subscr{\lambda}{min}) \bigr\} \bigr)}_{< 0}
 < 0.
\end{multline*}
In case that
$\min \bigl\{\arcsin(a + y_{1}\subscr{\lambda}{min}),\arcsin(b + y_{2} \subscr{\lambda}{min}) \bigr\} > -\gamma$, 
we obtain 
$\min\nolimits_{i \in \{3,4\}}\{\arcsin(x_{i} + y_{i}\subscr{\lambda}{min})\} = - \gamma$
and
\begin{multline*}
	f(\subscr{\lambda}{min}) <
\arcsin(x_{1} + y_{1}\subscr{\lambda}{min}) + \arcsin(x_{2} + y_{2} \subscr{\lambda}{min}) 
- \gamma \\
+  \max_{i \in \{3,4\}} \bigl\{ \arcsin(x_{i} + y_{i}\subscr{\lambda}{min}) \bigr\}\,.
\end{multline*}
%
Since $|x_{2}| < |x_{3}| \leq |x_{4}| < |x_{1}| \leq \sin(\gamma)$, it readily follows that
$\arcsin(x_{1} + y_{1}\subscr{\lambda}{min}) - \gamma < 0$ and
$\arcsin(x_{2} + y_{2} \subscr{\lambda}{min}) +  \max_{i \in \{3,4\}} \bigl\{ \arcsin(x_{i} + y_{i}\subscr{\lambda}{min}) \bigr\} < 0$. 
We conclude that
$f(\subscr{\lambda}{min}) < 0$.
Since $f(\subscr{\lambda}{min}) < 0 < f(0) \leq f(\subscr{\lambda}{max})$, it follows from Lemma \ref{Lemma: cycle feasibility} that there exists a stable equilibrium $\theta^{*} \in \bar\Delta_{G}(\gamma)$. The sufficiency is proved for $n = 4$.

{\it Proof of necessity for $n \in \{3,4\}$:} 
We prove the necessity by contradiction. Consider a compact cube $\mc Q = [-c,+c]^{|\mc E|} \subset \real^{|\mc E|}$, where $c> 0$ satisfies $c > \sin(\gamma) $. Assume that for every $x \in \fvec 1_n^\perp$, even those satisfying $\|x\|_{\infty} \geq c$, there exists $\lambda \in \mathbb R$ such that the cycle constraint $\fvec 1_{n}^{T}\arcsinbf(x + \lambda y) = 0$ and the norm constraint $\| x + \lambda y \|_{\infty} \leq \sin(\gamma)$ are simultaneously satisfied. For the sake of contradiction, consider now the symmetric case, where  $x \in \fvec 1_n^\perp$ has components $x_{i} \in \{-c,+c,0\}$. As proved in statement (C1), $\lambda^{*} =0$ uniquely solves the cycle constraint equation $0 = f(\lambda^{*} = 0) = \sum_{i=1}^{n} \arcsin(x_{i} + \lambda^{*}y) = \sum_{i=1}^{n} \arcsin(\pm c)$ for any value of $c \in {[0,1]}$. However, the norm constraint $\| x + \lambda^{*} y \|_{\infty} = \|x\|_{\infty} \leq \sin(\gamma)$ can be satisfied only if $\|x\|_{\infty} \leq \sin(\gamma) < c$. We arrive at a contradiction since we assumed $\|x\|_{\infty} \geq c > \sin(\gamma)$. 

We conclude that, if $x = B^{T} L^{\dagger} \omega$ is bounded within a compact cube $\mc Q = [-c,+c]^{|\mc E|} \subset \real^{|\mc E|}$ with $c \leq \sin(\gamma)$, the condition \eqref{eq-SI: sync condition - gamma} is also necessary for synchronization of all considered parametric realizations of $B^{T} L^{\dagger} \omega$ within this compact cube $\mc Q$. For the compact set $\fvec\Omega = \Omega^{n} \in \real^{n}$, it follows that the image $B^{T}L^{\dagger} \circ L\fvec\Omega = B^{T}\fvec\Omega$ equals the compact cube 
$\mc Q = \bigl[- \bigl(\max_{\omega \in \Omega} \omega - \min_{\omega \in \Omega} \omega\bigr) \,,\, +\bigl(\max_{\omega \in \Omega} \omega - \min_{\omega \in \Omega} \omega\bigr)\bigr]^{n}$. 
Hence, the condition \eqref{eq-SI: sync condition - gamma} is necessary for synchronization of all considered parametric realizations of $\omega$ in the compact set $L\fvec\Omega$. 
This concludes the proof of statement (C2).

{\em Statement (C3):} 
To prove the first part of statement (C3) we construct an explicit counterexample. Consider a cycle of length $n \geq 5$ with unit-weighed edges $a_{i,i+1} = 1$, and let
\begin{equation*}
	\omega = \alpha \cdot \begin{bmatrix} 1+\frac{1}{n-3} & 0 & -2 & 1-\frac{1}{n-3}  & \fvec 0_{n-4}\end{bmatrix}^{T} \,, 
\end{equation*} 
where $\alpha \in {[0,1]}$. For $\alpha<1$, these parameters satisfy the necessary conditions \eqref{eq-SI: necessary sync condition - 1} and \eqref{eq-SI: necessary sync condition - 2}. For the given parameters, we obtain the non-symmetric vector $x = B^{T}L^{\dagger} \omega$ given by
\begin{equation}
	x = B^{T}L^{\dagger} \omega = \alpha \cdot \begin{bmatrix} -1 & -1 & 1 & \frac{1}{n-3}\fvec 1_{(n-3)}  \end{bmatrix}^{T}
	\,.
\end{equation}
Notice that $\|x \|_{\infty} = \alpha < 1$, $x$ is non-symmetric, and $x$ is the minimum $\infty$-norm vector $\psi = x + \lambda \fvec 1_{n}$ for $\lambda \in \mathbb R$. 

In the following, we will show that there exists no equilibrium in $\lim_{\gamma \uparrow \pi/2} \bar\Delta_{G}(\gamma) = \bar\Delta_{G}(\pi/2)$. 
Consider the function $f(\lambda) = \arcsinbf(\fvec 1_{n}^{T}x + \lambda \fvec 1_{n})$ whose domain is centered symmetrically around zero, that is, $\subscr{\lambda}{max} = - \subscr{\lambda}{min} =  \lim_{\gamma \uparrow \pi/2} (\sin(\gamma) - \alpha) = 1 - \alpha$. Notice that the domain of $f$ vanishes as $\alpha \uparrow 1$.
 For $n \to \infty$ we have that 
$\lim_{n \to \infty} f(0) = 
	-\arcsin( \alpha) + \lim_{n \to \infty} (n-3) \cdot \arcsin(\alpha/(n-3)) = -\arcsin(\alpha) + \alpha
$.
Hence, as $n \to \infty$ and $\alpha \uparrow 1$, we obtain $f(0) = - \frac{\pi}{2}+1 < 0$. Due to continuity of $f$ with respect to $\alpha,n,\lambda$, we conclude that for $n \geq 5$ sufficiently large and $\alpha < 1$ sufficiently large%
, there is no $\lambda^{*}$ such that $f(\lambda^{*}) = 0$. Hence, the condition $\| x\|_{\infty} = \|B^{T} L^{\dagger} \omega\|_{\infty} < 1$ does generally not guarantee existence of $\theta^{*} \in \bar\Delta_{G}(\pi/2) \supset \Delta_{G}(\pi/2)$.
A second numerical counterexample will be constructed in Example \ref{Example: cyclic counter example} below.

A sufficient condition for the existence of an equilibrium $\theta^{*} \in \Delta_{G}(\gamma)$ is $x_{i} + \subscr{\lambda}{min} y_{i} \leq 0 \leq x_{i} + \subscr{\lambda}{max} y_{i}$ for each $i \in \until n$, which is equivalent to condition \eqref{eq-SI: sufficient condition for cyclic equilibrium}. Indeed if condition \eqref{eq-SI: sufficient condition for cyclic equilibrium} holds, we obtain $f(\subscr{\lambda}{min}) = \sum_{i=1}^{n} \arcsin(x_{i} + \subscr{\lambda}{min} y_{i})$ as a sum of nonpositive terms and $f(\subscr{\lambda}{max}) = \sum_{i=1}^{n} \arcsin(x_{i} + \subscr{\lambda}{max} y_{i})$ as a sum of nonnegative terms. 
Since $\fvec 1_{n}^{T}x = 0$ and generally $x \neq \fvec 0_{n}$ (otherwise we fall back in the symmetric case), at least one $x_{i}$ is strictly negative and at least one $x_{i}$ is strictly positive, and it follows that $f(\subscr{\lambda}{min}) < 0 < f(\subscr{\lambda}{max})$. The statement (C3) follows then immediately from Lemma \ref{Lemma: cycle feasibility}. 
This concludes the proof.
\end{proof}

In the following, define a {\em patched network} $\{G(\mc V,\mc
E,A),\omega\}$ as a collection of subgraphs and natural frequencies $\omega
\in \fvec 1_n^\perp$, where {\em(i)} each subgraph is connected, {\em(ii)}
in each subgraph one of the conditions (G1),(G2),(G3),(G4), (C1), or (C2)
is satisfied, {\em(iii)} the subgraphs are connected to another through
edges $\{i,j\} \in \mc E$ satisfying $\| (e_{|\mc E|}^{i} - e_{|\mc
  E|}^{j})^{T} L^{\dagger} \omega\|_{\infty} \leq \sin(\gamma)$, and
{\em(iv)} the set of cycles in the overall graph $G(\mc V,\mc E,A)$ is
equal to the union of the cycles of all subgraphs. Since a patched network
satisfies the synchronization condition \eqref{eq-SI: sync condition - gamma}
as well the norm and cycle constraints, we can state the following result.

\begin{corollary}{\bf(Sync condition for a patched network)}
\label{Corollary: Sync condition for patched graphs}
Consider the Kuramoto model \eqref{eq-SI: Kuramoto model} with a patched network $\{G(\mc V,\mc E,A),\omega\}$, and let $\gamma \in {[0,\pi/2[}$. There is an exponentially stable equilibrium $\theta^{*} \in \bar\Delta_{G}(\gamma)$ if condition \eqref{eq-SI: sync condition - gamma} holds.
\end{corollary}

\begin{example}{\bf(Numerical cyclic counterexample and its intuition)}
\label{Example: cyclic counter example}
\normalfont
In the proof of Theorem \ref{Theorem: cycle graph}, we provided an analytic counterexample which demonstrates that condition \eqref{eq-SI: sync condition - gamma} is not sufficiently tight for synchronization in sufficiently large cyclic networks. Here, we provide an additional numerical counterexample. Consider a cycle family of length $n = 5 + 3 \cdot p$, where $p \in \mathbb N_{0}$ is a nonnegative integer.
Without loss of generality, assume that the edges are labeled by $\{i,i+1\} \pmod n$ for $i \in \until n$ such that $ \Ker(B) = \textup{span}(\fvec 1_{n})$.
Assume that all edges are unit-weighed $a_{i,i+1\pmod n} = 1$ for $i \in \until n$. Consider $\alpha \in {[0,1[}$, and let 
\begin{equation*}
	\omega = \alpha \cdot \begin{bmatrix} -1/2 & 2 & \fvec 0_{p+1} & 3/2 &  \fvec 0_{2p+1}\end{bmatrix}^{T} \,. 
\end{equation*} 
For $n=5$ $(p=1)$ the graph and the network parameters are illustrated in Figure \ref{Fig: cyclic counter example}. 
For the given network parameters, we obtain the non-symmetric vector $B^{T}L^{\dagger} \omega$ given by
\begin{equation*}
	B^{T}L^{\dagger} \omega = \alpha \cdot \begin{bmatrix} 1 & -1 & -\fvec 1_{(n-2)/3} & 1/2 \cdot \fvec 1_{2(n-2)/3} \end{bmatrix}^{T}
	\,.
\end{equation*}
Analogously to the example provided in the proof of Theorem \ref{Theorem: cycle graph}, $\|B^{T} L^{\dagger} \omega \|_{\infty} = \alpha$ and $B^{T}L^{\dagger} \omega$ is the minimum $\infty$-norm vector $B^{T}L^{\dagger} \omega + \lambda \fvec 1_{n}$ for $\lambda \in \mathbb R$.
In the limit $\alpha \uparrow 1$, the necessary condition \eqref{eq-SI: necessary sync condition - 1} is satisfied with equality. In Figure \ref{Fig: cyclic counter example}, for $\alpha \uparrow 1$, we have that $\omega_{2} = 2$, and the necessary condition \eqref{eq-SI: necessary sync condition - 1} reads as $a_{12} + a_{23} = |\omega_{2}| = 2$, and the corresponding equilibrium equation $\sin(\theta_{1} - \theta_{2}) + \sin(\theta_{3} - \theta_{2}) = 2$ can only be satisfied if $\theta_{1} - \theta_{2} = \pi/2$ and $\theta_{3} - \theta_{2} = \pi/2$. Thus, with two fixed edge differences there is no more ``wiggle room'' to compensate for the effects of $\omega_{i}$, $i \in\{1,3,4,5\}$. As a consequence, there is no equilibrium $\theta^{*} \in \bar\Delta_{G}(\pi/2)$ for $\alpha = 1$ or equivalently $\|B^{T} L^{\dagger} \omega \|_{\infty} = 1$. 
Due to continuity of the equations \eqref{eq-SI: fixed-point equations - component-wise} with respect to $\alpha$, we conclude that for $\alpha < 1$ sufficiently large there is no equilibrium either.
Numerical investigations show that this conclusion is true, especially for very large cycles. For the extreme case $p = 10^{7}$, we obtain the critical threshold $\alpha \approx 0.9475$ where $\theta^{*}\in \bar\Delta_{G}(\pi/2)$ ceases to exist.
\oprocend
\begin{figure}[htbp]
	\centering{
	\includegraphics[width = 0.7\columnwidth]{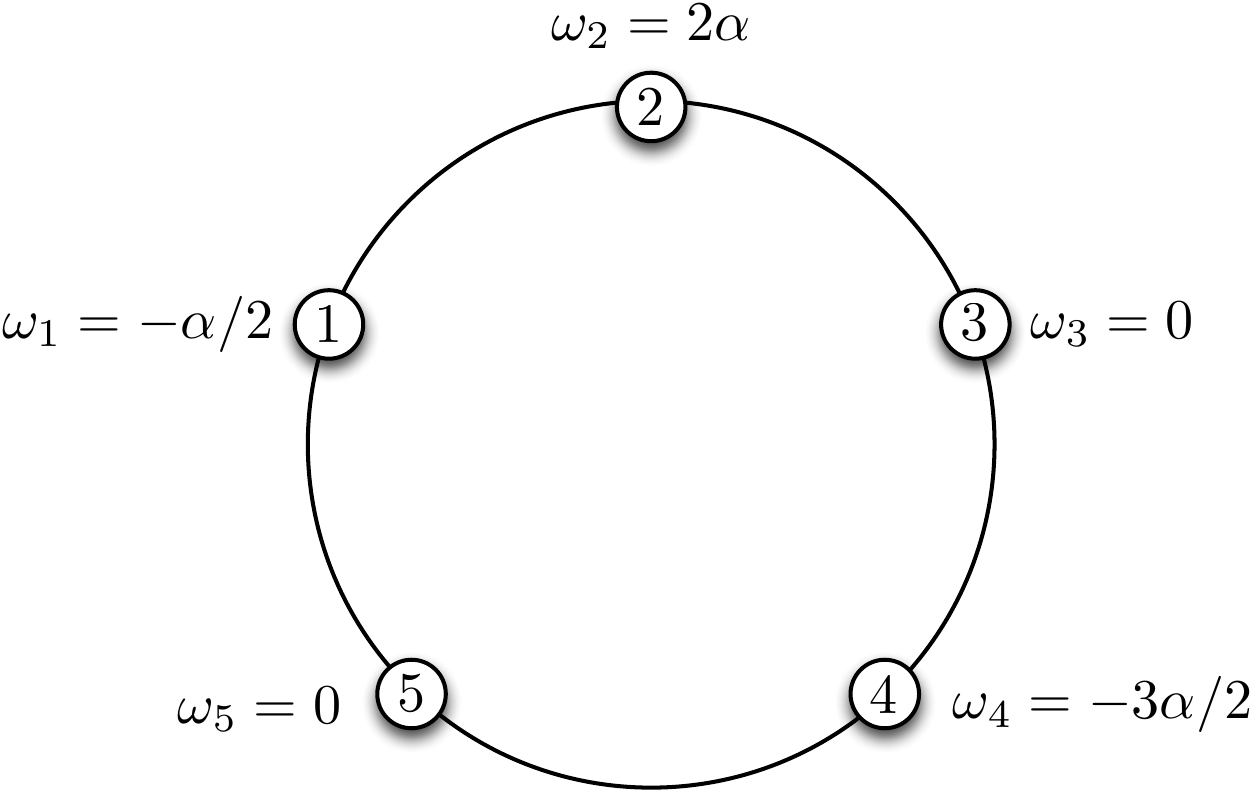}
	\caption{Cycle graph with $n=5$ nodes and non-symmetric choice of $\omega$.}
	\label{Fig: cyclic counter example}
	}
\end{figure}
\end{example}

Notice that both the counterexample used in the proof of Theorem \ref{Theorem: cycle graph} and the one in Example \ref{Example: cyclic counter example} are at the boundary of the admissible parameter space, where the necessary condition \eqref{eq-SI: necessary sync condition - 1} is marginally satisfied. In the next section, we establish that such ``degenerate'' counterexamples do almost never occur for generic network topologies and parameters.

To conclude this section, we remark that the main technical difficulty in proving sufficiency of the condition \eqref{eq-SI: sync condition - gamma} for arbitrary graphs is the compact state space $\mathbb T^{n}$ and the non-monotone sinusoidal coupling among the oscillators. Indeed, if the state space was $\mathbb R^{n}$ and if the oscillators were coupled via non-decreasing and odd functions, then the synchronization problem simplifies tremendously and the counterexamples in the proof of Theorem \ref{Theorem: cycle graph} and in Example \ref{Example: cyclic counter example} do not occur; see \citesec{MB-DZ-FA:11} for an elegant analysis based on optimization theory.

\section{Statistical Synchronization Assessment}
\label{Section: Statistical Evaluations}

After having established that the synchronization condition \eqref{eq-SI: sync condition - gamma} is necessary and sufficient for particular network topologies and parameters, we now validate both its correctness and its accuracy for arbitrary networks.

\subsection{Statistical Assessment of Correctness}

%
Extensive simulation studies lead us to the conclusion that condition \eqref{eq-SI: sync condition - gamma} is correct in general and guarantees the existence of a stable equilibrium $\theta^{*} \in \bar\Delta_{G}(\gamma)$. In order to validate this hypothesis we invoke probability estimation through Monte Carlo techniques, see \citesec[Section 9]{RT-GC-FD:05} and \citesec[Section 3]{GCC-FD-RT:11} for a comprehensive review.

We consider the following {\em nominal random networks} $\{G(\mc V,\mc E, A),\omega\}$ parametrized by the number $n \geq 2$ of nodes, the width $\alpha>0$ of the sampling region for each natural frequency $\omega_{i}$ and $i \in \until n$, and a connected random graph model $\textup{RGM}(p) = G(\mc V,\mc E(p))$ with node set $\mc V = \until n$ and edge set $\mc E = \mc E(p)$ induced by a coupling parameter $p \in {[0,1]}$. In particular, given the four parameters $(n,\textup{RGM},p,\alpha)$, a nominal random network is constructed as follows:

\begin{enumerate}

	
	\item[{\em (i)}] {\em Network topology:} To construct the network topology, we consider three different one-parameter families of random graph models $\textup{RGM}(p) = G(\mc V,\mc E(p))$, each parameterized by the number of nodes $n \geq 2$ and a coupling parameter $p \in {[0,1]}$. Specifically, we consider (i) an Erd\"os-R\'enyi random graph model (\textup{RGM} = \textup{ERG}) with probability $p$ of connecting two nodes, 
(ii) a random geometric graph model (\textup{RGM} = \textup{RGG}) with sampling region $[0,1]^{2} \subset\mathbb R^{2}$, connectivity radius $p$, and 
(iii) a Watts-Strogatz small world network (\textup{RGM} = \textup{SMN}) \citesec{DJW-SHS:98} with initial coupling of each node to its two nearest neighbors and rewiring probability $p$. 
If, for a given $n \geq 2$ and $p \in [0,1]$, the realization of a random graph model is not connected, then this realization is discarded and new realization is constructed;

	\item[{\em (ii)}] {\em Coupling weights:} For a given random graph $G(\mc V,\mc E(p))$, for each edge $\{i,j\} \in \mc E(p)$, the coupling weight $a_{ij}=a_{ji}>0$ is sampled from a uniform distribution supported on the interval ${[0.5,5]}$; 
	
	\item[{\em (iii)}] {\em Natural frequencies:} For a given $n \geq 2$ and $\alpha>0$, the natural frequencies $\omega \in \fvec 1_{n}^{\perp}$ are constructed in two steps. In a first step, $n$ real numbers $q_{i}$, $i \in \until n$, are sampled from a uniform distribution supported on ${[-\alpha/2,+\alpha/2]}$, where $\alpha>0$. In a second step, by subtracting the average $\sum_{i=1}^{n} q_{i}/n$ we define $\omega_{i} = q_{i} - \sum_{i=1}^{n} q_{i}/n$ for $i \in \until n$ and obtain $\omega = (\omega_{1},\dots,\omega_{n}) \in \fvec 1_{n}^{\perp}$; and 
	
	\item[{\em (iv)}]  {\em Parametric realizations:} We consider forty realizations of the parameter 4-tuple $(n,\textup{RGM},p,\alpha)$ covering a wide range of network sizes $n$, coupling parameters $p$, and natural frequencies $\omega$, which are listed in the first column of Table \ref{Table: Monte Carlo Simulations}. The choices of $\alpha$ in these forty cases is such that the resulting equilibrium angles $\theta^*$ satisfy on average $\max_{\{i,j\} \in \mc E} | \theta_{i}^{*} - \theta_{j}^{*}| \approx \pi/3$.
	
\end{enumerate}
%
%
For each of the forty parametric realizations in {\em (iv)}, we generate 30000 nominal models of $\omega \in \fvec 1_{n}^{\perp}$ and $G(\mc V,\mc E, A)$ (conditioned on connectivity) as detailed in {\em (i) - (iii)} above, each satisfying $\|B^{T} L^{\dagger} \omega \|_{\infty} < 1$. If a sample does not satisfy $\|B^{T} L^{\dagger} \omega \|_{\infty} < 1$, it is discarded and a new sample is generated. Hence, we obtain $1.2 \cdot 10^{6}$ nominal random networks $\{G(\mc V,\mc E, A),\omega\}$, each with a connected graph $G(\mc V,\mc E, A)$ and $\omega \in \fvec 1_{n}^{\perp}$ satisfying $\|B^{T} L^{\dagger} \omega \|_{\infty} \leq \sin(\gamma)$ for some $\gamma < \pi/2$. 

For each case and each instance, we numerically solve equation \eqref{eq-SI: fixed-point equations - vector form} with accuracy $10^{-6}$ and test the hypothesis 
\begin{equation*}
	\mc H:\, \left\| B^{T} L^{\dagger} \omega \right\|_{\infty} \leq \sin(\gamma) \;\implies\; \exists\; \theta^{*} \in \bar\Delta_{G}(\gamma)
\end{equation*} 
with an accuracy $10^{-4}$. The results are reported in Table \ref{Table: Monte Carlo Simulations} together with the empirical probability that the hypothesis $\mc H$ is true for a set of parameters $(n,\textup{RGM},p,\alpha)$. Given a set of parameters $(n,\textup{RGM},p,\alpha)$ and 30000 samples, the empirical probability is calculated as
\begin{equation*}
\widehat{ \mbox{\tt{Prob}} }_{(n,\textup{RGM},p,\alpha)} \!=\! \frac{\mbox{number of samples satisfying}\, \bigl( \mc H \mbox{\tt{\,is\,true}}  \bigr)}{ 30000 }
\,.
\end{equation*}
Given an accuracy level $\epsilon \in {]0,1[}$ and a confidence level $\eta \in {]0,1[}$, we ask for the number of samples $N$ such that the true probability $\mbox{\tt{Prob}}_{(n,\textup{RGM},p,\alpha)} \bigl( \mbox{\tt $\mc H$\,is\,true} \bigr)$ equals the empirical probability $\widehat{ \mbox{\tt{Prob}} }_{(n,\textup{RGM},p,\alpha)}$ with confidence level greater  than $1-\eta$ and accuracy at least $\epsilon$, that is, 
\begin{multline*}
	\mbox{\tt{Prob}} \left(
	\bigl|
	\mbox{\tt{Prob}}_{(n,\textup{RGM},p,\alpha)} \bigl( \mbox{\tt $\mc H$\,is\,true} \bigr) - \widehat{ \mbox{\tt{Prob}} }_{(n,\textup{RGM},p,\alpha)}
	\bigr| < \epsilon
	\right)\\
	> 1-\eta
	\,.
\end{multline*}
By the Chernoff bound \citesec[Equation (9.14)]{RT-GC-FD:05}, the number of samples $N$ for a given accuracy $\epsilon$ and confidence $\eta$ is given as
\begin{equation}
	N \geq \frac{1}{2 \epsilon^{2}} \log \frac{2}{\eta}.
	\label{eq-SI: Chernoff bound}
\end{equation}
For $\epsilon = \eta = 0.01$, the Chernoff bound \eqref{eq-SI: Chernoff bound} is satisfied for $N \geq 26492$ samples.
By invoking the Chernoff bound \eqref{eq-SI: Chernoff bound}, our simulations studies establish the following statement:
\begin{quotation}
\noindent
With 99\% confidence level, there is at least 99\% accuracy that the hypothesis $\mc H$ is true with probability 99.97 \% for a nominal network constructed as in {\em (i) - (iv)} above.\\

\noindent
In particular, for a nominal network with parameters $(n,\textup{RGM},p,\alpha)$ constructed as in {\em (i) - (iv)} above, with 99\% confidence level, there is at least 99\% accuracy that the probability $\mbox{\tt{Prob}}_{(n,\textup{RGM},p,\alpha)} \bigl( \mbox{\tt $\mc H$\,is\,true} \bigr)$  equals the empirical probability $\widehat{ \mbox{\tt{Prob}} }_{(n,\textup{RGM},p,\alpha)}$, as listed in Table~\ref{Table: Monte Carlo Simulations}, that is,
\end{quotation}
\begin{multline*}
	\mbox{\tt{Prob}} \Bigl(
	\bigl|
	\mbox{\tt{Prob}}_{(n,\textup{RGM},p,\alpha)} \bigl( \mbox{\tt $\mc H$ is true} \bigr) \\
	- \widehat{ \mbox{\tt{Prob}} }_{(n,\textup{RGM},p,\alpha)}
	\bigr| < 0.01
	\Bigr)
	> 0.99
	\,.
\end{multline*}
It can be seen in Table \ref{Table: Monte Carlo Simulations} that for large and dense networks the hypothesis $\mc H$ is always true, whereas for small and sparsely connected networks the hypothesis $\mc H$
can marginally fail with an error of
order $\mc O(10^{-4})$. Thus, for these cases a tighter condition of the form $\| B^{T}
L^{\dagger} \omega \|_{\infty} \!\leq\! \sin(\gamma) - \mc O(10^{-4})$ is
required to establish the existence of $\theta^{*} \in
\bar\Delta_{G}(\gamma)$. These results indicate that ``degenerate'' 
topologies and parameters (such as the large and isolated cycles used in the proof of Theorem \ref{Theorem: cycle graph} and in Example \ref{Example: cyclic counter example}) are more likely to occur in small networks. 

\subsection{Statistical Assessment of Accuracy}

As established in the previous subsection, the synchronization condition \eqref{eq-SI: sync condition - gamma} is a scalar synchronization test with predictive 
power for almost all network topologies and parameters. This remarkable fact is difficult to establish via statistical studies in the vast parameter space.
Since we proved in statement (G4) of Theorem \ref{Theorem: Exact and tight synchronization conditions for extremal graphs} that condition \eqref{eq-SI: sync condition - gamma} is exact for sufficiently small pairwise phase cohesiveness $| \theta_{i} - \theta_{j} | \ll 1$ (or equivalently, for sufficiently identical natural frequencies $\omega_{i}$ and sufficiently strong coupling), we investigate the other extreme $\max_{\{i,j\} \in \mc E} | \theta_{i} - \theta_{j} | = \pi/2$. To test the corresponding synchronization condition \eqref{eq-SI: sync condition} in a low-dimensional parameter space, we consider a complex network of Kuramoto oscillators
\begin{equation}
	\dot \theta_{i}
	=
	\omega_{i} - \,K \cdot  \sum\nolimits_{j=1}^{n}  a_{ij} \, \sin(\theta_{i}-\theta_{j})
	\,, \quad i \in \until n \,,
	\label{eq-SI: topological Kuramoto model}
\end{equation}
where $K>0$ is the coupling gain among the oscillators and the coupling weights are assumed to be unit-weighted, that is, $a_{ij} = a_{ji} = 1$ for all $\{i,j\} \in \mc E$.
If $L$ is the unweighted Laplacian matrix, then
 condition \eqref{eq-SI: sync condition} reads as $K > \subscr{K}{critical}
\triangleq \| L^{\dagger} \omega \|_{\mc E,\infty}$. Of course, the condition $K > \subscr{K}{critical}$ is only sufficient and synchronization may occur for a smaller value of $K$  than $\subscr{K}{critical}$. 
In order to test the accuracy of the condition $K > \subscr{K}{critical}$, 
 we numerically found the smallest value of $K$ leading to synchrony for various network sizes,
connected random graph models, and sample distributions of the natural frequencies. 
Here we discuss in detail the construction of the random network topologies and parameters leading to the data displayed in Figure 3 of the main manuscript.

We consider the following {\em nominal random networks} $\{G(\mc V,\mc E, A),\omega\}$ parametrized by the number of nodes $n \in \{10,20,40,160\}$, the sampling distribution \textup{SD} for the natural frequencies $\omega  \in \fvec 1_{n}^{\perp}$, and a connected random graph model $\textup{RGM}(p) = G(\mc V,\mc E(p))$ with node set $\mc V = \until n$ and edge set $\mc E = \mc E(p)$ induced by a coupling parameter $p \in {[0,1]}$. In particular, given the four parameters $(n,\textup{RGM},p,\textup{SD})$, a nominal random network is constructed as follows:

\begin{enumerate}
	
	\item[{\em (i)}] {\em Network topology and weights:}  To construct the network topology, we consider three different one-parameter families of random graph models $\textup{RGM}(p) = G(\mc V,\mc E(p))$, each parameterized by the number of nodes $n$ and a coupling parameter $p \in {[0,1]}$. Specifically, we consider (i) an Erd\"os-R\'enyi random graph model (\textup{RGM} = \textup{ERG}) with probability $p$ of connecting two nodes, 
(ii) a random geometric graph model (\textup{RGM} = \textup{RGG}) with sampling region $[0,1]^{2} \subset\mathbb R^{2}$, connectivity radius $p$, and 
(iii) a Watts-Strogatz small world network (\textup{RGM} = \textup{SMN}) \citesec{DJW-SHS:98} with initial coupling of each node to its two nearest neighbors and rewiring probability $p$. 
If, for a given $n$ and $p \in [0,1]$, the realization of a random graph model is not connected, then this realization is discarded and new realization is constructed. All nonzero coupling weights are set to one, that is, $a_{ij}=a_{ji}=1$ for $\{i,j\} \in \mc E$; 
	
	\item[{\em (ii)}] {\em Natural frequencies:} For a given network size $n$ and sampling distribution \textup{SD}, the natural frequencies $\omega \in \fvec 1_{n}^{\perp}$ are constructed in three steps. In a first step, the sampling distribution of the natural frequencies is chosen. For classic Kuramoto oscillators with uniform coupling $a_{ij} = K/n$ for distinct $i,j \in \until n$, we know that the two extreme sampling distributions (with bounded support) are the bipolar discrete and the uniform distribution leading to the largest and smallest critical coupling, respectively~\citesec{FD-FB:10w}. 
	Here we choose a  uniform distribution ($\textup{SD} = \textup{uniform}$) supported on ${[-1,+1]}$ or a bipolar discrete distribution ($\textup{SD} = \textup{bipolar}$) supported on ${\{-1,+1\}}$. In a second step, $n$ real numbers $q_{i}$, $i \in \until n$, are sampled from the distribution \textup{SD}. In a third step, by subtracting the average $\sum_{i=1}^{n} q_{i}/n$ we define $\omega_{i} = q_{i} - \sum_{i=1}^{n} q_{i}/n$ for $i \in \until n$ and obtain $\omega = (\omega_{1},\dots,\omega_{n}) \in \fvec 1_{n}^{\perp}$; and
		
	\item[{\em (iii)}]  {\em Parametric realizations:} We consider 600 realizations of parameter 4-tuple $(n,\textup{RGM},p,\textup{SD})$ covering a wide range of network sizes $n$, coupling parameters $p$, and natural frequencies $\omega$. All 600 realizations are shown in Figure 3 in the main manuscript.
	
\end{enumerate}

For each of the 600 parametric realizations in {\em (iii)}, we generate 100 nominal models of $\omega \in \fvec 1_{n}^{\perp}$ and $G(\mc V,\mc E, A)$ (conditioned on connectivity) as detailed in {\em (i) - (ii)} above. Hence, we obtain 60000 nominal random networks $\{G(\mc V,\mc E, A),\omega\}$, each with a connected graph $G(\mc V,\mc E, A)$ and natural frequencies $\omega \in \fvec 1_{n}^{\perp}$.
For each sample network, we consider the complex Kuramoto model \eqref{eq-SI: topological Kuramoto model} and numerically find the smallest value of $K$ leading to synchrony with cohesive phases satisfying $\max_{\{i,j\} \in \mc E} | \theta_{i} - \theta_{j}| = \pi/2$. The critical value of $K$ is found iteratively by integrating the Kuramoto dynamics \eqref{eq-SI: topological Kuramoto model} and decreasing $K$ if the steady state $\theta^{*}$ satisfies $\max_{\{i,j\} \in \mc E} | \theta_{i}^{*} - \theta_{j}^{*}| < \pi/2$ and increasing $K$ otherwise. We repeat this iteration until a steady state $\theta^{*}$ is found satisfying $\max_{\{i,j\} \in \mc E} | \theta_{i} - \theta_{j}| = \pi/2$ with accuracy $10^{-3}$.
Our findings are reported in Figure 3 in the main manuscript, where each data point corresponds to the sample mean of 100 nominal models with the same parameter 4-tuple $(n,\textup{RGM},p,\textup{SD})$.

\section{Synchronization Assessment for Power Networks}
\label{Section: Simulation Data for RTS 96}

We envision that our proposed condition \eqref{eq-SI: sync condition - gamma} can be applied to quickly assess synchronization and robustness in power networks under volatile operating conditions. 
%
Since real-world power networks are carefully engineered systems with particular network topologies and parameters, 
they cannot be reduced to the standard topological random graph models \citesec{ZW-AS-RJT:10}, and 
we do not extrapolate the statistical results from the previous section to power grids. Rather, we consider ten widely-established and commonly studied IEEE power network test cases provided by \citesec{RDZ-CEM-DG:11,CG-PW-PA-RA-MB-RB-QC-CF-SH-SK-WL-RM-DP-NR-DR-AS-MS-CS:99} to validate the correctness and the predictive power of our synchronization condition \eqref{eq-SI: sync condition - gamma}.

\subsection{Statistical Synchronization Assessment for IEEE Systems}

We validate the synchronization condition \eqref{eq-SI: sync condition - gamma} in a smart power grid scenario subject to fluctuations in load and generation and equipped with fast-ramping generation and controllable demand. Here, we report the detailed simulation setup leading to the results shown in Table 1 of the main manuscript.

The nominal simulation parameters for the ten IEEE test cases can be found in \citesec{RDZ-CEM-DG:11,CG-PW-PA-RA-MB-RB-QC-CF-SH-SK-WL-RM-DP-NR-DR-AS-MS-CS:99}. 
Under nominal operating conditions, the power generation is optimized to
meet the forecast demand, while obeying the AC power flow laws and
respecting the thermal limits of each transmission line. Thermal limits
constraints are precisely equivalent to phase cohesiveness requirements,
that is, for each line $\{i,j\}$, the angular distance $| \theta_{i} - \theta_{j}|$ 
needs to be bounded such that the corresponding power flow 
$a_{ij}\sin(\theta_{i}-\theta_{j})$ is bounded.
Here, we found the optimal generator power injections through the standard optimal power flow solver provided by {\em MATPOWER} \citesec{RDZ-CEM-DG:11}.

In order to test the synchronization condition \eqref{eq-SI: sync condition - gamma} in a volatile smart grid scenario, we make the following changes to the nominal IEEE test cases with optimal generation: 

\begin{enumerate}

	\item[(i)] {\em Fluctuating loads with stochastic power demand:} We assume fluctuating demand and randomize 50\% of all loads (selected independently with identical distribution) to deviate from the forecasted loads with Gaussian statistics (with nominal power injection as mean and standard deviation 0.3 in per unit system);
	
		\item[(ii)] {\em Renewables with stochastic power generation:} We assume that the grid is penetrated by renewables with severely fluctuating power outputs, for example, wind or solar farms, and we randomize 33\% of all generating units (selected independently with identical distribution) to deviate from the nominally scheduled generation with Gaussian statistics (with nominal power injection as mean and standard deviation 0.3 in per unit system); and
		
		\item[(iii)] {\em Fast-ramping generation and controllable loads:} Following the paradigm of {\em smart operation of smart grids} \citesec{PPV-FFW-JWB:11}, the fluctuations can be mitigated by fast-ramping generation, such as fast-response energy storage including batteries and flywheels, and controllable loads, such as large-scale server farms or fleets of plug-in hybrid electrical vehicles. Here, we assume that the grid is equipped with 10\% fast-ramping generation (10\% of all generators, selected independently with identical distribution) and 10\% controllable loads (10\% of all loads, selected independently with identical distribution), and the power imbalance (caused by fluctuating demand and generation) is uniformly dispatched among these adjustable power sources.

\end{enumerate}

For each of the ten IEEE test cases with optimal generator power injections, we construct 1000 random realizations of the scenario (i)-(iii) described above. For each realization, we numerically check for the existence of a solution $\theta^{*} \in \bar \Delta_{G}(\gamma)$, $\gamma \in {[0,\pi/2[}$ to the AC power flow equations, the right-hand side of the power network dynamics \eqref{eq-SI: generator dynamics}-\eqref{eq-SI: bus dynamics}, given~by
\begin{align}
	\begin{split}
	P_{\textup{m},i} 
	&= 
	\sum\nolimits_{j=1}^{n} a_{ij} \sin(\theta_{i} - \theta_{j})
	\,,\;\;\;  i \in \mc V_{1} \,,\\
	 P_{\textup{l},i}
	&=
	- \sum\nolimits_{j=1}^{n} a_{ij} \sin(\theta_{i} - \theta_{j})
	\,,\;\;\; i \in \mc V_{2} \,.
	\end{split}
	\label{eq-SI: AC power flow}
\end{align}
The solution to the AC power flow equations \eqref{eq-SI: AC power flow} is found via the AC power flow solver provided by {\em MATPOWER} \citesec{RDZ-CEM-DG:11}.  Notice that, by Lemma \ref{Lemma: stable fixed point in pi/2 arc}, if such a solution $\theta^{*}$ exists, then it is unique (up to rotational invariance) and also locally exponentially stable with respect to the power network dynamics \eqref{eq-SI: generator dynamics}-\eqref{eq-SI: bus dynamics}.
Next, we compare the numerical solution $\theta^{*}$ with the results predicted by our synchronization condition \eqref{eq-SI: sync condition - gamma}.
As discussed in Remark \ref{Remark: interpretation of sync condition}, a physical insightful and computationally efficient way to evaluate condition \eqref{eq-SI: sync condition - gamma} is to solve the sparse and linear DC power flow equations given by 
\begin{align}
	\begin{split}
	P_{\textup{m},i} 
	&= 
	\sum\nolimits_{j=1}^{n} a_{ij} (\delta_{i} - \delta{j})
	\,,\;\;\;  i \in \mc V_{1} \,,\\
	 P_{\textup{l},i}
	&=
	- \sum\nolimits_{j=1}^{n} a_{ij} (\delta_{i} - \delta_{j})
	\,,\;\;\; i \in \mc V_{2} \,.
	\end{split}
	\label{eq-SI: DC power flow}
\end{align}
The solution $\delta^{*}$ of the DC power flow equations \eqref{eq-SI: DC power flow} is defined uniquely up to the usual translational invariance. 
Given the solution $\delta^{*}$ of the DC power flow equations \eqref{eq-SI: DC power flow}, the left-hand side of our synchronization condition \eqref{eq-SI: sync condition - gamma} evaluates to $\| B^{T} L^{\dagger} \omega \|_{\infty} = \| L^{\dagger} \omega \|_{\mc E,\infty} = \max_{\{i,j\} \in \mc E} | \delta_{i}^{*} - \delta_{j}^{*}|$.

Finally, we compare our prediction with the numerical results. If $\| B^{T} L^{\dagger} \omega \|_{\infty} \leq \sin(\gamma)$ for some $\gamma \in {[0,\pi/2[}$, then condition \eqref{eq-SI: sync condition - gamma} predicts that there exists a stable solution $\theta \in \bar \Delta_{G}(\gamma)$, or alternatively $\theta \in \bar \Delta_{G}(\arcsin(\| B^{T} L^{\dagger} \omega \|_{\infty}))$. To validate this hypothesis, we compare the numerical solution $\theta^{*}$ to the AC power flow equations \eqref{eq-SI: AC power flow} with our prediction $\theta^{*} \in \bar \Delta_{G}(\arcsin(\| B^{T} L^{\dagger} \omega \|_{\infty}))$.
Our findings and the detailed statistics are reported in Table 1 of the main manuscript. It can be observed that condition \eqref{eq-SI: sync condition - gamma} predicts the correct phase cohesiveness $| \theta_{i}^{*} - \theta_{j}^{*}|$ along all transmission lines $\{i,j\} \in \mc E$ with extremely high accuracy even for large-scale networks,  such as the Polish power grid model featuring 2383 nodes.

\subsection{Simulation Data for IEEE Reliability Test System~96}

The IEEE Reliability Test System 1996 (RTS 96) is a widely adopted and relatively large-scale power network test case, which has been designed as a benchmark model for power flow and stability studies.
The RTS 96 is a multi-area model featuring 40 load buses and 33 generation buses, as illustrated in Figure 4 in the main manuscript. 
The network parameters and the dynamic generator parameters can be found in \citesec{CG-PW-PA-RA-MB-RB-QC-CF-SH-SK-WL-RM-DP-NR-DR-AS-MS-CS:99}. 

%
The quantities $a_{ij}$ in the coupled oscillator model \eqref{eq-SI: coupled oscillator model} correspond to the product of the voltage magnitudes at buses $i$ and $j$ as well the susceptance of the transmission line connecting buses $i$ and $j$. For a given set of power injections at the buses and branch parameters, the voltage magnitudes and initial phase angles were calculated using the optimal power flow solver provided by {\em MATPOWER} \citesec{RDZ-CEM-DG:11}. 
The quantities $\omega_{i}$, $i \in \mc V_{2}$, are the real power demands at loads, and $\omega_{i}$, $i \in \mc V_{1}$, are the real power injections at the generators, which were found through the optimal power flow solver provided by {\em MATPOWER} \citesec{RDZ-CEM-DG:11}.
We made the following changes in order to adapt the detailed RTS 96 model to the classic structure-preserving power network model \eqref{eq-SI: generator dynamics}-\eqref{eq-SI: bus dynamics} describing the generator rotor and voltage phase dynamics. First, we replaced the synchronous condenser in the original RTS 96 model \citesec{CG-PW-PA-RA-MB-RB-QC-CF-SH-SK-WL-RM-DP-NR-DR-AS-MS-CS:99} by a U50 hydro generator.
Second, since the numerical values of the damping coefficients $D_{i}$ are not contained in the original RTS 96 description \citesec{CG-PW-PA-RA-MB-RB-QC-CF-SH-SK-WL-RM-DP-NR-DR-AS-MS-CS:99}, we chose the following values to be found in \citesec{PK:94}: for the generator damping, we chose the uniform damping coefficient $D_{i} = 1$ in per unit system and for $i \in \mc V_{1}$, and for the load frequency coefficient we chose $D_{i} = 0.1 \textup{\,s}$ for $i \in \mc V_{2}$. 
Third and finally, we discarded an optional high voltage DC link for the branch $\{113,316\}$.

\subsection{Bifurcation Scenario in the IEEE Reliability Test System 96}
 
As shown in the main manuscript, an imbalanced power dispatch in the RTS 96 network together with a tripped generator (generator 323) in the Southeastern (green) area results in a loss of synchrony since the maximal power transfer is limited due to thermal constraints. This loss of synchrony can be predicted by our synchronization condition \eqref{eq-SI: sync condition - gamma} with extremely high accuracy.
In the following, we show that a similar loss of synchrony occurs, even if the generator 323 is not disconnected and there are no thermal limit constraints on the transmission lines. In this case, the loss of synchrony is due to a saddle node bifurcation at an inter-area angle of $\pi/2$, which can be predicted accurately by condition \eqref{eq-SI: sync condition - gamma} as~well.

For the following dynamic simulation we consider again an imbalanced power dispatch: the demand at each load in the Southeast (green) area is increased by a uniform amount and the resulting power imbalance is compensated by uniformly increasing the generation at each generator in the two Western (blue) areas. 
The imbalanced power dispatch essentially transforms the RTS 96 into a two-oscillator network, and we observe the classic loss of synchrony through a saddle-node bifurcation \citesec{FD-FB:10w,ID:92} shown in Figures \ref{Fig: RTS 96 -- stable} and \ref{Fig: RTS 96 -- unstable}.
In particular, the network is still synchronized for a load increase of 141\% resulting in $\left\| L^{\dagger} \omega \right\|_{\mc E,\infty} = 0.9995  < 1$. If the loads are increased by an additional 10\% resulting in $\left\| L^{\dagger} \omega \right\|_{\mc E,\infty} = 1.0560 > 1$, then synchronization is lost and the areas separate via the transmission lines $\{121, 325\}$ and $\{223, 318\}$.
%
Of course, in real-world power networks
the transmission lines $\{121, 325\}$ and $\{223, 318\}$ would be separated at some smaller inter-area angle $\gamma^{*} \ll \pi/2$ due to thermal limits. For instance, the transmission line $\{121,325\}$ is separated at the angle $\gamma^{*} = 0.1978 \pi$ 
corresponding to a 78\% load increase, which can be predicted from condition \eqref{eq-SI: sync condition} as $\gamma^{*} \approx \arcsin(\| L^{\dagger} \omega\|_{\mc E,\infty})$ with an accuracy of $0.0023\pi$. 
In summary, this transmission line  scenario illustrates the accuracy of the proposed condition \eqref{eq-SI: sync condition}.

\begin{figure*}
	\centering{
	\includegraphics[width=1.9\columnwidth]{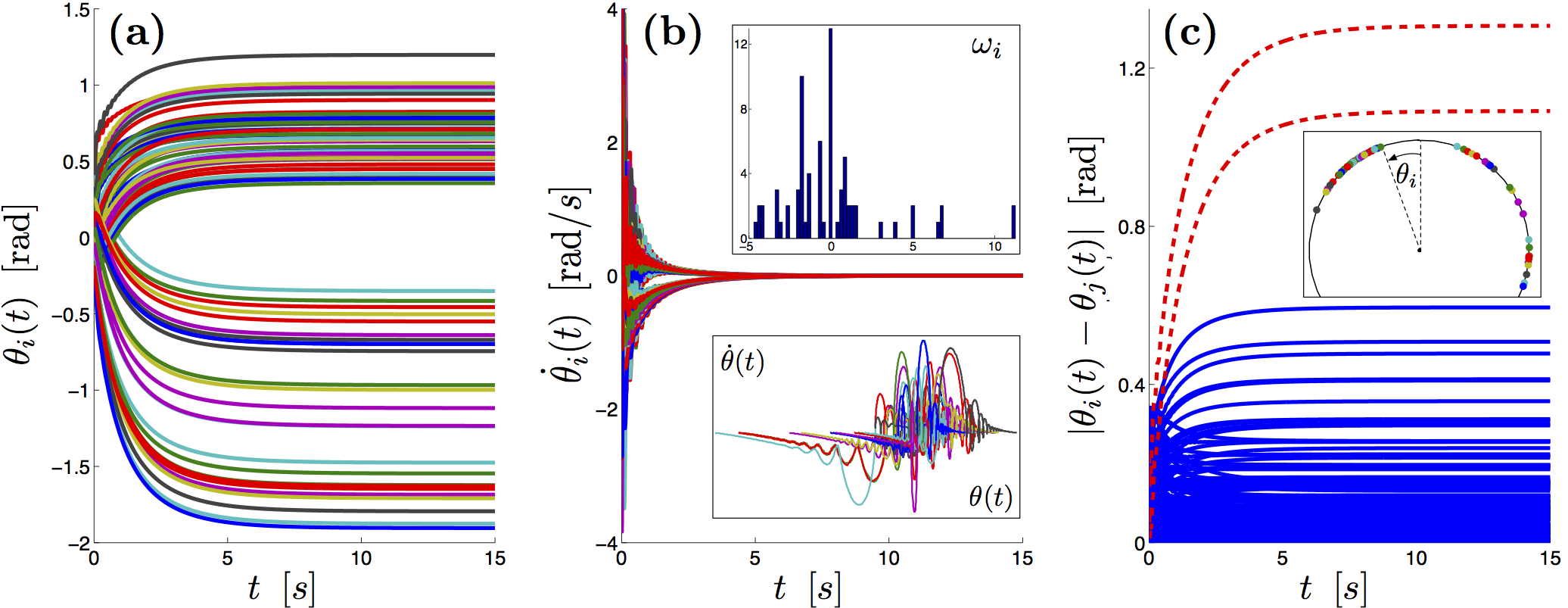}
	\caption{Time series of the RTS 96 dynamics for 141\% load increase resulting in $\| B^{T} L^{\dagger} \omega \|_{\infty} = \| L^{\dagger} \omega \|_{\mc E,\infty} = 0.9995  < 1$. Figure {\bf (a)} depicts the angles $\theta_{i}(t)$, Figure {\bf (b)} shows the frequencies $\dot \theta_{i}(t)$, and Figure {\bf (c)} depicts the angular distances $| \theta_{i}(t) - \theta_{j}(t)|$ over transmission lines, where the red dashed curves correspond to the pairs $\{121,325\}$ and $\{223,318\}$. The inserts show the power injections $\omega_{i}$, the phase space of the generator dynamics $(\theta(t),\dot\theta(t))$, and the stationary angles~$\theta_{i}$.}
	\label{Fig: RTS 96 -- stable}
	}
\end{figure*}

\begin{figure*}
	\centering{
	\includegraphics[width=1.9\columnwidth]{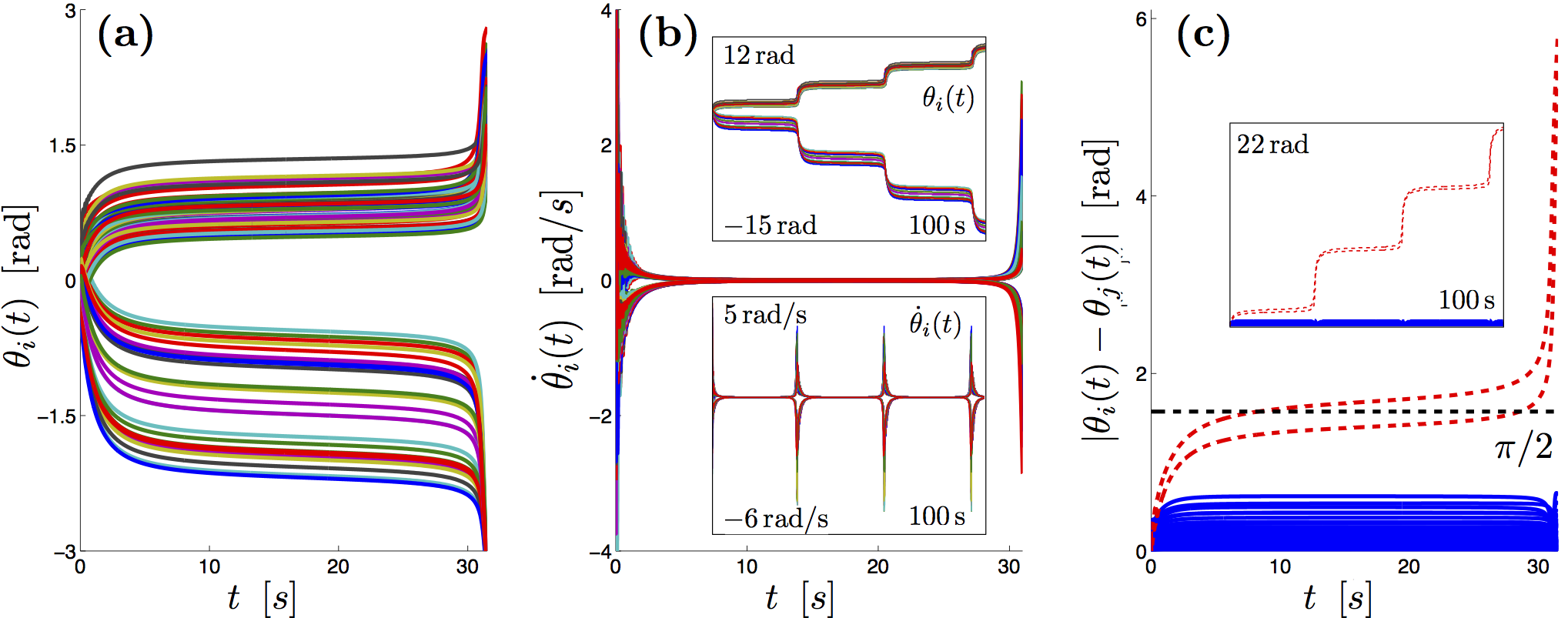}
	\caption{Time series of the RTS 96 dynamics for 151\% load increase resulting in $B^{T} L^{\dagger} \omega \|_{\infty} = \|  L^{\dagger} \omega \|_{\mc E,\infty} = 1.0560 > 1$. Figure {\bf (a)} depicts the angles $\theta_{i}(t)$, Figure {\bf (b)} depicts the frequencies $\dot \theta_{i}(t)$, and Figure {\bf (c)} depicts the angular distances $| \theta_{i}(t) - \theta_{j}(t)|$ over transmission lines, which diverge for the pairs $\{121,325\}$ and $\{223,318\}$ shown as red dashed curves. The inserts depict the long-time dynamics simulated over 100s.}
	\label{Fig: RTS 96 -- unstable}
	}
\end{figure*}


\begin{table}[h]
\caption{Results of the Monte Carlo simulations to test the hypothesis $\mc H$.}
\label{Table: Monte Carlo Simulations}
\end{table}
\vspace{-2cm}

\begin{table*}[h]
     \caption{Results of the Monte Carlo simulations to test the hypothesis $\mc H$.}
  \begin{tabular}{ | l || c | c |l}
    \hline
    { $\Bigl.$ 
   nominal random network}  & 
    {failures of hypothesis $\mc H$:} & 
     {empirical probability: } \\
     {parametrized by $(n,\textup{RGM},p,\alpha)$ }& 
      { \#\;$\bigl( \mc H \mbox{\tt{ is not true}}  \bigr)$ } & 
       {$\widehat{ \mbox{\tt{Prob}} }_{(n,\textup{RGM},p,\alpha)} \Bigr.$ }
    \\\hline\hline
     $(10,\textup{ERG},0.15,6)$ & 104 & 99.653 \% 
     \\\hline
    $(10,\textup{ERG},0.3,8)$ & 65 & 99.783 \% 
        \\\hline
    $(10,\textup{ERG},0.5,14)$ & 15 & 99.950 \%
     \\\hline
    $(10,\textup{ERG},0.75,25)$ & 0 & 100 \%
        \\\hline
     $(20,\textup{ERG},0.15,10)$ & 80 & 99.733 \% 
    \\\hline
     $(20,\textup{ERG},0.3,15)$ & 5 & 99.983 \% 
          \\\hline
    $(20,\textup{ERG},0.5,24)$ & 0 & 100 \% 
          \\\hline
    $(20,\textup{ERG},0.75,45)$ & 0 & 100 \% 
          \\\hline
    $(30,\textup{ERG},0.15,13)$ & 22 & 99.927 \% 
     \\\hline
    $(30,\textup{ERG},0.3,20)$ & 0 & 100 \%
       \\\hline
    $(30,\textup{ERG},0.5,37)$ & 0 & 100 \%
            \\\hline
    $(30,\textup{ERG},0.75,65)$ & 0 & 100 \%
      \\\hline
    $(60,\textup{ERG},0.15,20)$ & 1 & 99.997 \% 
          \\\hline
    $(60,\textup{ERG},0.3,40)$ & 0 & 100 \%
          \\\hline
    $(60,\textup{ERG},0.5,70)$ & 0 & 100 \%
                \\\hline
    $(60,\textup{ERG},0.75,125)$ & 0 & 100 \%
          \\\hline
    $(120,\textup{ERG},0.15,35)$ & 0 & 100 \%
          \\\hline
    $(120,\textup{ERG},0.3,75)$ & 0 & 100 \%
               \\\hline
    $(120,\textup{ERG},0.5,130)$ & 0 & 100 \%
       \\\hline
    $(120,\textup{ERG},0.75,235)$ & 0 & 100 \%
        \\\hline\hline
    $(10,\textup{RGG},0.3,10)$ & 15 & 99.950 \%
    \\\hline
    $(10,\textup{RGG},0.5,15)$ & 18 & 99.940 \%
    \\\hline
    $(20,\textup{RGG},0.3,10)$ & 23 & 99.924 \%
    \\\hline
    $(20,\textup{RGG},0.5,15)$ & 3 & 99.990 \%
    \\\hline
    $(30,\textup{RGG},0.3,10)$ & 31 & 99.897 \%
    \\\hline
    $(30,\textup{RGG},0.5,15)$ & 1 & 99.997 \%
    \\\hline
    $(60,\textup{RGG},0.3,10)$ & 3 & 99.990 \%
    \\\hline
    $(60,\textup{RGG},0.5,15)$ & 0 & 100 \%
    \\\hline
    $(120,\textup{RGG},0.3,10)$ & 0 & 100 \%
     \\\hline
    $(120,\textup{RGG},0.5,15)$ & 0 & 100 \%
    \\\hline\hline
    $(10,\textup{SMN},0.1,10)$ & 2 & 99.994 \%
    \\\hline
    $(10,\textup{SMN},0.2,10)$ & 0 & 100 \%
    \\\hline
    $(20,\textup{SMN},0.1,13)$ & 0 & 100 \%
    \\\hline
    $(20,\textup{SMN},0.2,13)$ & 0 & 100 \%
    \\\hline
    $(30,\textup{SMN},0.1,10)$ & 0 & 100 \%
    \\\hline
    $(30,\textup{SMN},0.2,13)$ & 0 & 100 \%
    \\\hline
    $(60,\textup{SMN},0.1,7)$ & 0 & 100 \%
        \\\hline
    $(60,\textup{SMN},0.2,7)$ & 0 & 100 \%
    \\\hline
    $(120,\textup{SMN},0.1,4)$ & 0 & 100 \%
        \\\hline
    $(120,\textup{SMN},0.2,4)$ & 0 & 100 \%
     \\\hline\hline
     { $\Bigl.$ 
     over all $1.2 \cdot 10^{6}$ instances } & 
     { 388} & 
     { 99.968 \% }
     \\\hline
  \end{tabular}
  \tablenote{}{Overall, $1.2 \cdot 10^{6}$ instances of $\{G(\mc V,\mc E, A),\omega\}$ were constructed as described in {\em (i) - (iv)} above, each satisfying $\|B^{T} L^{\dagger} \omega \|_{\infty} < 1$. For each instance, the fixed-point equation \eqref{eq-SI: fixed-point equations - vector form} was solved with accuracy $10^{-6}$, and the failures of the hypothesis $\mc H$ were reported within an accuracy of $10^{-4}$, that is, failures of order $10^{-5}$ were discarded.}
\end{table*}



\end{article}

\end{document}